\documentclass{amsart}
\usepackage{amssymb}
\usepackage[backref=page]{hyperref}
\usepackage{enumerate}
\newtheorem{theorem}{Theorem}[section]
\newtheorem{definition}[theorem]{Definition}
\newtheorem{corollary}[theorem]{Corollary}
\newtheorem{remark}{Remark}
\newtheorem{lemma}[theorem]{Lemma}
\numberwithin{equation}{section}

\DeclareMathOperator{\supp}{supp}
\DeclareMathOperator{\spn}{span}
\DeclareMathOperator{\dist}{dist}

\DeclareFontFamily{U}{mathx}{\hyphenchar\font45}
\DeclareFontShape{U}{mathx}{m}{n}{
      <5> <6> <7> <8> <9> <10>
      <10.95> <12> <14.4> <17.28> <20.74> <24.88>
      mathx10
      }{}
\DeclareSymbolFont{mathx}{U}{mathx}{m}{n}
\DeclareFontSubstitution{U}{mathx}{m}{n}
\DeclareMathAccent{\widecheck}{0}{mathx}{"71}
\DeclareMathAccent{\wideparen}{0}{mathx}{"75}

\title[Exceptional sets for length under projection]{Exceptional sets for length under restricted families of projections onto lines in $\mathbb{R}^3$}
\author{Terence L.~J.~Harris}
\address{Department of Mathematics\\ University of Wisconsin\\ 480
Lincoln Drive\\ Madison\\ WI\\ 53706\\ USA}
\curraddr{School of Mathematics and Physics\\ The University of Queensland\\ St Lucia QLD 4067\\  Australia}
\email{terry.harris@uq.edu.au}

\begin{document} 
\begin{abstract} It is shown that if $A \subseteq \mathbb{R}^3$ is a Borel set of Hausdorff dimension $\dim A>1$, and if $\rho_{\theta}$ is orthogonal projection to the line spanned by $( \cos \theta, \sin \theta, 1 )$, then $\rho_{\theta}(A)$ has positive length for all $\theta$ outside a set of Hausdorff dimension at most $\frac{3-\dim A}{2}$.   \end{abstract}
\maketitle
\section{Introduction}

Let $I$ be a compact subinterval of $\mathbb{R}$, and let $\gamma: I \to S^2$ be a $C^2$ curve with $\det\left( \gamma, \gamma', \gamma''\right)$ nonvanishing on $I$. Let $\rho_{\theta}$ be orthogonal projection onto the span of $\gamma(\theta)$, given by 
\[ \rho_{\theta}(x) = \langle x, \gamma(\theta) \rangle \gamma(\theta). \]
The model example is when $\gamma(\theta) = \frac{1}{\sqrt{2}} \left( \cos \theta, \sin \theta, 1 \right)$ and $I= [0, 2\pi]$, but everything stated below holds in the general case. The main result of this work is the following theorem (see Subsection~\ref{notation} for notation). 

\begin{theorem} \label{projmeasure} Let $\alpha>1$, and suppose that $\mu$ is a compactly supported Borel measure on $\mathbb{R}^3$ with $c_{\alpha}(\mu) < \infty$. Then 
\[ \dim\left\{ \theta \in I : \rho_{\theta\sharp} \mu \not\ll \mathcal{H}^1 \right\} \leq \frac{3-\alpha}{2}. \] \end{theorem} 

The following slightly stronger inequality will be shown via the same method of proof.

\begin{theorem} \label{Lptheorem} If $\alpha >1$, then for any $\beta > \frac{3-\alpha}{2}$, there exists a $p= p(\alpha,\beta, \gamma)$ with $1 < p \leq 2$ and a constant $C = C(\alpha, \beta, \gamma)$ such that the following holds. If $\mu$ is a Borel measure supported in $B_3(0,1)$ such that $c_{\alpha}(\mu) < \infty$, and if $\lambda$ is a Borel measure supported in $I$ such that $c_{\beta}(\lambda) < \infty$,  then 
\[ \int_I\left\lVert \rho_{\theta\sharp} \mu \right\rVert_{L^p(\mathcal{H}^1)}^p \, d\lambda(\theta) \leq C c_{\alpha}(\mu)^{p-1} \mu(\mathbb{R}^3) \lambda(\mathbb{R})^{\frac{2-p}{2}} c_{\beta}(\lambda)^{\frac{p}{2}}.  \] \end{theorem}

By a theorem of Mattila~\cite[Theorem~1.2]{mattila}, a corollary of Theorem~\ref{Lptheorem} is the following.
\begin{corollary} \label{corollaryintersection} Suppose that $t>1$, and that $A \subseteq \mathbb{R}^3$ is $\mathcal{H}^t$-measurable with $0 < \mathcal{H}^t(A) < \infty$. Then there is a set $E \subseteq I$ with 
\[ \dim E \leq \frac{3-t}{2}, \]
such that for any $\theta \in I \setminus E$, 
\[ \mathcal{H}^1\left\{ y \in \spn \gamma(\theta) : \dim\left(\rho_{\theta}^{-1}(y) \cap A\right) = t-1 \right\}>0. \] \end{corollary}
Corollary~\ref{corollaryintersection} is an exceptional set version of (one half of) Theorem~1.4 from \cite{greenharrisou}. By Frostman's lemma and a theorem of Davies~\cite{davies} that a Borel (or analytic) set of dimension strictly greater than $t$ (or even $\mathcal{H}^t(A) = \infty$) contains an $\mathcal{H}^t$-measurable subset of positive finite $\mathcal{H}^t$ measure, a consequence of Corollary~\ref{corollaryintersection} is the following. 
\begin{corollary} \label{corollaryborel} If $A \subseteq \mathbb{R}^3$ is Borel (or analytic) and $\dim A > 1$, then there is a set $E \subseteq I$ with 
\[ \dim E \leq \frac{3-\dim A}{2}, \]
such that for any $\epsilon>0$ and $\theta \in I \setminus E$, 
\begin{equation} \label{fibrelength} \mathcal{H}^1\left\{ y \in \spn \gamma(\theta) : \dim\left(\rho_{\theta}^{-1}(y) \cap A\right) \geq \dim A-1-\epsilon \right\}>0. \end{equation}
If $\mathcal{H}^t(A)>0$ where $\dim A = t$, then \eqref{fibrelength} holds with $\epsilon =0$. 
\end{corollary}
By taking $\epsilon$ in \eqref{fibrelength} positive but strictly smaller than $\dim A -1$, a consequence of  Corollary~\ref{corollaryborel} is the following. 
\begin{corollary} \label{lengthcorollary} If $A \subseteq \mathbb{R}^3$ is Borel (or analytic) and $\dim A > 1$, then 
\begin{equation} \label{lengthexception} \dim \left\{ \theta \in I :  \mathcal{H}^1( \rho_{\theta}(A) ) =0 \right\} \leq \frac{3-\dim A}{2}. \end{equation}
\end{corollary}
Rather than going through intersections, Corollary~\ref{lengthcorollary} also follows directly from Theorem~\ref{projmeasure} by Frostman's lemma and the definition of absolute continuity. Corollary~\ref{lengthcorollary} can also be deduced easily from Theorem~\ref{Lptheorem}. Theorem~\ref{projmeasure} also implies (see the proof of Theorem~1.3 in \cite{harris24}) that if $A \subseteq \mathbb{R}^3$ is $\mathcal{H}^t$-measurable with $0 < \mathcal{H}^t(A) < \infty$ for some $t>1$, then there is a Borel set $E_{\theta}$ with $\dim E_{\theta} \leq (3-t)/2$, such that any $\mathcal{H}^t$-measurable set $B \subseteq A$ with $\mathcal{H}^t(B)>0$ satisfies $\mathcal{H}^1(\rho_{\theta}(B))>0$ for all $\theta \in I \setminus E_{\theta}$.

The previous exceptional set inequality for the set in \eqref{lengthexception} was $\frac{4-\dim A}{3}$, from \cite{harris24}. However, with the zero length condition $\mathcal{H}^1( \rho_{\theta}(A) )=0$ replaced by the slightly stronger condition $\dim(\rho_{\theta}(A)) < 1$, Gan, Guth, and Maldague proved in~\cite{ganguthmaldague} that
\begin{equation} \label{ganguthmaldagueinequality} \dim \left\{ \theta \in I :  \dim\left( \rho_{\theta}(A) \right) < 1 \right\} \leq \frac{3- \dim A}{2}. \end{equation}
Therefore, Corollary~\ref{lengthcorollary} is a refinement of \eqref{ganguthmaldagueinequality} which weakens the requirement $\dim(\rho_{\theta}(A)) < 1
$ to just $\rho_{\theta}(A)$ having zero length (see also Remark~3 in \cite{ganguthmaldague} for a brief discussion). In \cite{ganguthmaldague}, they actually showed more generally\footnote{Because the right-hand side of \eqref{ganguthmaldagueinequality2} is unchanged if $t$ and $s$ are both increased by the same amount, \eqref{ganguthmaldagueinequality2} is a consequence of \eqref{ganguthmaldagueinequality}; projection theorems for lower dimensional sets follow from projection theorems for higher dimensional sets by taking a sumset of $A$ with a random lower dimensional set (see e.g.~\cite[Proposition~A.1]{fasslerorponen}). There is also a simple proof (just using Frostman's lemma and the definition of Hausdorff dimension) that Theorem~\ref{Lptheorem} implies \eqref{ganguthmaldagueinequality2}.} that if $0 < s \leq 1$ and $A \subseteq \mathbb{R}^3$ is a Borel set with $\dim A = t$, then
\begin{equation} \label{ganguthmaldagueinequality2} \dim \left\{ \theta \in I :  \dim\left( \rho_{\theta}(A) \right) < s \right\} \leq \max\left\{ 0 , \frac{2- (t-s)}{2} \right\}. \end{equation}

\begin{sloppypar} The above inequalities are often referred to as ``Falconer-type'' exceptional set bounds, because Falconer proved~\cite{falconer} that 
\[ \dim\left\{ L \in G(n,1) : \mathcal{H}^1(\pi_L(A)) = 0 \right\} \leq n- \dim A, \] whenever $A$ is a Borel subset of $\mathbb{R}^n$ with $\dim A >1$, and $\pi_L$ is orthogonal projection to the line $L$. Returning to the restricted projection problem in $\mathbb{R}^3$, for $t=s$, Pramanik,~Yang, and Zahl \cite{pramanikyangzahl} proved that if $A$ is a Borel set with $\dim A =t < 1$, then 
\[ \dim \left\{ \theta \in I :  \dim\left( \rho_{\theta}(A) \right) < t \right\} \leq t, \]
which is better than \eqref{ganguthmaldagueinequality2} when $t$ is close to $s$ (more precisely when $t< 2-s$), and which is usually referred to as a ``Kaufman-type'' exceptional set bound (after \cite{kaufman}). An earlier exceptional set bound was proved by Käenmäki, Orponen, and Venieri in \cite{KOV}, which solved the almost everywhere restricted projection problem for the curve $\gamma(\theta) = \frac{1}{\sqrt{2}} \left( \cos \theta, \sin \theta, 1 \right)$, but their exceptional set bound is weaker than \eqref{ganguthmaldagueinequality2} for all $t>s$.\end{sloppypar}

\subsection{Notes on the method}

The main improvement to the exceptional set inequality over the one from \cite{harris24} comes from using a ``small cap'' wave packet decomposition instead of a standard wave packet decomposition, for the cone. This decomposes the cone in frequency space into small caps (or tubes) instead of standard caps (or planks), and correspondingly decomposes physical space into slabs dual to these small caps (instead of planks). This is more natural for the projections $\rho_{\theta}$, because the inverse of a $\delta$-interval under $\rho_{\theta}$ is a $\delta$-slab (a $\delta$-neighbourhood of a ``light plane''). As in \cite{harris24}, the approach to proving Theorem~\ref{projmeasure} is via a ``good-bad'' decomposition of the measure $\mu$ (adapting arguments from \cite{GIOW,liu} originally set up for the distance set problem), and using small caps makes the bound on the ``bad'' part less wasteful. Unfortunately, the approach to bounding the ``good'' part in \cite{harris24} used refined decoupling, which does not seem to work well with small caps. Instead, the ``good'' part is converted back into a standard wave packet decomposition for the cone over planks $T$. For a given plank $T$ inside a slab $S$, there are either very few other planks $T'$ inside $S$ with $\mu(T) \sim \mu(T')$, or the $\mu$-measure of $T$ must be much smaller than $\mu(S)$. In either of these situations there is a gain in the argument over the one in \cite{harris24}, and a pigeonholing trick is used to make the argument precise and extract the optimal gain. This pigeonholing trick and the switching between wave packet decompositions (from around \eqref{intermezzo3} to \eqref{uncertainty3}) is the main novelty of the argument, and the key difference from \cite{harris24}. 

The argument from \cite{ganguthmaldague} also used small caps, and used (fractal) $L^4$ small cap decoupling for the cone.  Small cap decoupling is not used here, but only standard $L^6$ (refined) decoupling for the cone, so although there are some analogies to the argument in \cite{ganguthmaldague}, and the use of small caps is inspired by \cite{ganguthmaldague}, the two methods are substantially different. Since there is not yet a known small cap decoupling for the 2-dimensional cone generated by the moment curve in higher dimensions (as far as I am aware), the approach used here might be more generalisable to higher dimensions. See~\cite{ganguowang} for the proof of the (almost everywhere) Hausdorff dimension version of the higher dimensional restricted projection problem.  

The exponent $p>1$ obtained in Theorem~\ref{Lptheorem} is extremely close to 1 and tends to 1 as $\alpha \to 1^+$; because the argument is formulated to optimise the bound on the exceptional set, which is unaffected by the value of $p$. The argument in \cite{greenharrisou} obtains a $p$ arbitrarily close to $3/2$ (thus bounded away from 1 as $\alpha \to 1^+$), but only if the measure $\lambda$ supported in $I$ is the Lebesgue measure, and the use of the Lebesgue measure in \cite{greenharrisou} is very important. Therefore, the two methods do not seem to be comparable. 

\subsection{Notation} \label{notation}

\begin{enumerate}[(i)] \item For a measure $\mu$ on $X$ and a measurable function $f: X \to Y$, the notation $f_{\sharp} \mu$ denotes the pushforward of $\mu$ under $f$, which is a measure on $Y$ defined by 
\[ (f_{\sharp}\mu)(E) = \mu(f^{-1}(E) ), \]
for any measurable set $E \subseteq Y$. Pushforwards are defined similarly for finite complex measures. 

\item The $s$-dimensional Hausdorff measure on Euclidean space will be denoted by $\mathcal{H}^s$. In particular, $\mathcal{H}^3$ denotes Lebesgue measure on $\mathbb{R}^3$, and $\mathcal{H}^1$ on any line in $\mathbb{R}^3$ is the Lebesgue measure on that line. 

\item For a Borel measure $\mu$ on $\mathbb{R}^n$ and $\alpha >0$, 
\[ c_{\alpha}(\mu) := \sup_{x \in \mathbb{R}^n, r >0 } \frac{ \mu(B(x,r) ) }{r^{\alpha}}. \]

\item Given two measures $\mu$ and $\nu$ on the same $\sigma$-algebra, the notation $\mu \ll \nu$ means that $\mu$ is absolutely continuous with respect to $\nu$, meaning that $\mu(A) = 0$ whenever $A$ is measurable and $\nu(A)=0$. The notation $\mu \not\ll \nu$ means that $\mu$ is not absolutely continuous with respect to $\nu$. 

\item For a subset $A$ of Euclidean space, $\dim A$ denotes the Hausdorff dimension of $A$. 

\item For a compactly supported Borel measure $\mu$ on $\mathbb{R}^n$ (not necessarily absolutely continuous with respect to Lebesgue measure), and $p \in (1,\infty)$, the $L^p$ norm of $\mu$ is defined to be 
\[ \|\mu \|_p = \sup_{\substack{ \|f\|_{p'} \leq 1 \\ f: \mathbb{R}^n \to \mathbb{C} \text{ continuous } } } \left\lvert \int f \, d\mu \right\rvert, \]
where $\frac{1}{p} + \frac{1}{p'}=1$. The quantity $\|\rho_{\theta\sharp} \mu \|_{L^p(\mathcal{H}^1)}$ is defined by identifying $\rho_{\theta}$ with $x \mapsto \langle x, \gamma(\theta) \rangle$ and considering the corresponding measure on $\mathbb{R}$. Note that since $C(K)$ is separable for compact $K$, the supremum above may be restricted to a countable dense subset, and thus $\|\rho_{\theta\sharp} \mu \|_{L^p(\mathcal{H}^1)}$ is Borel measurable in $\theta$.  \end{enumerate} 

\subsection*{Acknowledgements} 

I thank the referee for carefully reading the manuscript and providing many helpful comments and suggestions. 

\section{Proof of Corollary~\ref{corollaryintersection} from Theorem~\ref{Lptheorem}}

\begin{sloppypar} \begin{proof}[Proof that Theorem~\ref{Lptheorem} implies Corollary~\ref{corollaryintersection}] 
Let $t>1$, let $A \subseteq \mathbb{R}^3$ be $\mathcal{H}^t$-measurable with $0 < \mathcal{H}^t(A) < \infty$. By a lemma of Besicovitch and Moran (see \cite[Lemma~3]{besicovitchmoran} or \cite[Theorem~1]{KR}), there exists an $F_{\sigma}$ set $A' \subseteq A$ with $\mathcal{H}^t(A' ) =\mathcal{H}^t(A)$. Therefore, by the upper bound in the Marstrand slice theorem \cite[Theorem~7.7]{mattila1995}, it may be assumed that $A$ is Borel. Let
\[ E = \left\{ \theta : \mathcal{H}^1\left\{ y \in \spn \gamma(\theta) : \dim\left( \rho_{\theta}^{-1}(y) \cap A\right) = t-1 \right\} = 0 \right\}. \]
By \cite[Theorem~7.7]{mattila1995}, for every $\theta$, 
\[  \dim\left( \rho_{\theta}^{-1}(y) \cap A\right) \leq t-1, \qquad \mathcal{H}^1\text{ a.e.~} y \in \spn \gamma(\theta) \]
and hence 
\[ E = \left\{ \theta : \mathcal{H}^1\left\{ y \in \spn \gamma(\theta) : \dim\left( \rho_{\theta}^{-1}(y) \cap A\right) \geq t-1 \right\} = 0 \right\}. \]
This gives 
\begin{equation} \label{Eformula3} E = \left\{ \theta : \int_{\spn(\gamma(\theta) )} \inf_{n \geq 1} \mathcal{H}^{t-1-\frac{1}{n}}\left( \rho_{\theta}^{-1}(y) \cap A \right) \, d\mathcal{H}^1(y) = 0 \right\}. \end{equation}
Since $A$ is Borel, the set 
\[ B= \left\{ (\theta, y, x) \in I \times \mathbb{R}^3 \times \mathbb{R}^3: x \in \rho_{\theta}^{-1}(y) \cap A \right\} \]
is Borel in $I \times \mathbb{R}^3 \times \mathbb{R}^3$. Therefore, by Tonelli's theorem \cite[Theorem 1.7.15]{tao}, for any $s \geq 0$ the function 
\[ (\theta,y) \mapsto  \mathcal{H}^s\left( \rho_{\theta}^{-1}(y) \cap A \right) = \int \chi_B(\theta,y,x) \, d\mathcal{H}^s(x) \]
is Borel measurable on $I \times \mathbb{R}$. By Tonelli's theorem \cite[Theorem 1.7.15]{tao} again, the function $\theta \mapsto \int_{\spn(\gamma(\theta) )} \inf_{n \geq 1} \mathcal{H}^{t-1-\frac{1}{n}}\left( \rho_{\theta}^{-1}(y) \cap A \right) \, d\mathcal{H}^1(y)$ is Borel measurable in $\theta$, and thus $E$ is a Borel set by \eqref{Eformula3}. 

Suppose for a contradiction that $\dim E > (3-t)/2$. Then by Frostman's lemma, there exists $\beta > (3-t)/2$ and a probability measure $\lambda$ on $E$ with $c_{\beta}(\lambda) < \infty$.
By Theorem~\ref{Lptheorem} and \cite[Theorem~1.2]{mattila},  for $\lambda$-a.e.~$\theta$, 
\begin{equation} \label{econtradiction} \mathcal{H}^1\left\{ y \in \spn \gamma(\theta) : \dim\left( \rho_{\theta}^{-1}(y) \cap A\right) = t-1 \right\} > 0. \end{equation} 
But $\lambda(E)=1$, so there must exist $\theta \in E$ such that \eqref{econtradiction} holds, contradicting the definition of $E$. 
\end{proof} \end{sloppypar}

\section{Small cap wave packet decomposition}

This section sets up the analogues of Lemma~2.2 and Lemma~2.3 from \cite{harris24} for the small cap wave packet decomposition. Throughout the next two sections, $\gamma: I \to S^2$ will be a fixed $C^2$ curve of unit speed (without loss of generality), with $\det(\gamma, \gamma', \gamma'')$ nonvanishing on $I$, where $I$ is a compact interval. The unit speed assumption implies that $\left\{ \gamma(\theta), \gamma'(\theta), \left( \gamma \times \gamma' \right)(\theta) \right\}$ is an orthonormal basis for $\mathbb{R}^3$, for any $\theta \in I$.

\begin{definition} \label{decompdefn} Let
\[ \Lambda = \bigcup_{j \geq 1} \Lambda_j, \]
where each $\Lambda_j$ is a collection of boxes (``small caps'') $\tau$ of dimensions $1 \times 1 \times 2^j$ with $\dist\left( \tau, \{0\}\right) \geq 2^{j-1}$, forming a boundedly overlapping cover of the $\sim 1$-neighbourhood of the truncated light cone $\Gamma_j$, where 
\[ \Gamma_j = \left\{ t \gamma(\theta) : 2^{j-1} \leq \lvert t \rvert \leq 2^j, \quad \theta \in I \right\}. \] 
Each $\tau \in \Lambda_j$ has an angle $\theta_{\tau}$ such that the long axis of $\tau$ is parallel to $\gamma(\theta_{\tau} )$. Let $\{\psi_{\tau}\}_{\tau \in \Lambda}$ be a smooth partition of unity subordinate to the cover $\left\{1.1\tau: \tau \in \Lambda\right\}$ of $\bigcup_{\tau \in \Lambda} \tau$, such that for any $\tau \in \Lambda_j$,  $\xi,v \in \mathbb{R}^3$, $t \in \mathbb{R}$ and any positive integer $n$,
\begin{multline} \label{derivconditiontau} \left\lvert \left( \frac{d}{dt} \right)^n \psi_{\tau}( \xi + t v ) \right\rvert \lesssim_n  \\
\left\lvert 2^{-j} \left\langle v, \gamma(\theta_{\tau}) \right\rangle \right\rvert^n + \left\lvert \left\langle v, \gamma'(\theta_{\tau}) \right\rangle \right\rvert^n  + \left\lvert \left\langle v, (\gamma \times \gamma')(\theta_{\tau}) \right\rangle \right\rvert^n. \end{multline}

Fix $\sigma>0$. For each $\tau \in \Lambda$, let $\mathbb{S}_{\tau}$ be a collection of boxes (referred to as ``slabs'') of dimensions $2^{j \sigma} \times 2^{j\sigma} \times 2^{-j\left(1-\sigma\right)}$, forming a boundedly overlapping cover of $\mathbb{R}^3$, such that for each $S \in \mathbb{S}_{\tau}$, the short direction of $S$ is parallel to $\gamma(\theta_{\tau})$. Let $\{\eta_S\}_{S \in \mathbb{S}_{\tau}}$ be a smooth partition of unity subordinate to the cover $\mathbb{S}_{\tau}$ of $\mathbb{R}^3$, such that for any $x,v \in \mathbb{R}^3$, any $t \in \mathbb{R}$ and any positive integer $n$, 
\begin{multline} \label{derivcondition} \left\lvert \left( \frac{d}{dt} \right)^n \eta_S( x + t v ) \right\rvert \lesssim_n  \\
\left\lvert 2^{j\left( 1- \sigma \right) } \left\langle v, \gamma(\theta_{\tau}) \right\rangle \right\rvert^n + \left\lvert 2^{-j\sigma} \left\langle v, \gamma'(\theta_{\tau}) \right\rangle \right\rvert^n  + \left\lvert 2^{-j \sigma} \left\langle v, (\gamma \times \gamma')(\theta_{\tau}) \right\rangle \right\rvert^n. \end{multline}

Given a finite complex Borel measure $\mu$ on $\mathbb{R}^3$, $\tau \in \Lambda$ and $S \in \mathbb{S}_{\tau}$, define
\[ M_S \mu = \eta_S\left( \mu \ast \widecheck{\psi_{\tau}} \right). \]  \end{definition}

The following non-stationary phase lemma will be important later; it is the small cap analogue of \cite[Lemma~2.2]{harris24}. 

\begin{lemma} \label{IBP} There exists an $r>0$, depending only on $\gamma$, such that the following holds. Let $j \geq 1$ and let $\tau \in \Lambda_j$. If $\theta \in I$ is such that $2^{-j\left(1- \sigma \right)} \leq \lvert \theta_{\tau} - \theta\rvert  \leq r$, then for any $S \in \mathbb{S}_{\tau}$, for any positive integer $N$ and for any finite complex Borel measure $\mu$,
\[ \left\lVert \rho_{\theta \sharp} M_S \mu \right\rVert_{L^1(\mathcal{H}^1)} \leq C 2^{-j\sigma N } \lvert \mu \rvert(\mathbb{R}^3), \]
where $C = C\left(N, \gamma, \sigma\right)$.  \end{lemma}
\begin{remark} The idea is that $\mu \ast \widecheck{\psi_{\tau}}$ is Fourier supported in $1.1\tau$, and therefore $M_S \mu = \eta_S\left( \mu \ast \widecheck{\psi_{\tau}} \right)$ is ``essentially'' Fourier supported in $1.11 \tau$. Since $\widehat{ \rho_{\theta\sharp} \mu} = \widehat{\mu} \circ \rho_{\theta}$, this means that the Fourier transform of $\rho_{\theta\sharp}  M_S \mu$ will be negligible when $2^{-j\left(1- \sigma \right)} \leq \lvert \theta_{\tau} - \theta\rvert \leq r$, since it is evaluating $\widehat{M_S \mu}$ outside its ``essential'' support. \end{remark}
\begin{proof} Since $\rho_{\theta}(S)$ has diameter $\lesssim 2^{j \sigma} \left\lvert \theta-\theta_{\tau} \right\rvert$, it suffices to prove that 
\[ \left\lVert \rho_{\theta \sharp} M_S \mu \right\rVert_{\infty} \lesssim 2^{-j\sigma N } \left\lvert \theta-\theta_{\tau} \right\rvert^{-1} \lvert \mu \rvert(\mathbb{R}^3). \]
 By the Hausdorff-Young inequality (in one dimension), it is enough to show that
\begin{multline*}  \int_{\mathbb{R}} \left\lvert \widehat{M_S \mu }\left( t \gamma(\theta ) \right) \right\rvert \, dt = \int_{\mathbb{R}} \left\vert \int_{\mathbb{R}^3} \widehat{\eta_S}( \xi ) \psi_{\tau}(t \gamma(\theta) -\xi) \widehat{\mu}(t \gamma(\theta) -\xi) \, d\xi \right\rvert \, dt\\
\lesssim 2^{-j\sigma N } \left\lvert \theta-\theta_{\tau} \right\rvert^{-1} \lvert \mu \rvert(\mathbb{R}^3). \end{multline*}
Since $\left\lVert \widehat{\mu} \right\rVert_{\infty} \leq \lvert \mu \rvert(\mathbb{R}^3)$, it suffices to prove that
\begin{equation} \label{statphase} \int_{\mathbb{R}} \int_{\mathbb{R}^3} \left\lvert \widehat{\eta_S}(\xi )\right\rvert \left\lvert \psi_{\tau}(t \gamma(\theta) -\xi) \right\rvert \, d\xi \, dt \lesssim 2^{-j\sigma N } \left\lvert \theta-\theta_{\tau} \right\rvert^{-1}. \end{equation}
Let $\widehat{S}$ be the box of dimensions $2^{-j \sigma} \times 2^{-j \sigma} \times 2^{j\left(1-\sigma \right)}$ around the origin dual to the slab $S$. If $\xi \in 2^{j\sigma/2}\widehat{S}$, it will be shown that $t \gamma(\theta) -\xi\notin 1.1\tau$ for any $t \in \mathbb{R}$, for any $\theta$ with $2^{-j\left(1- \sigma \right)} \leq \lvert \theta_{\tau} - \theta\rvert  \leq r$. Suppose for a contradiction that $t \gamma(\theta) -\xi \in 1.1\tau$ for some $t \in \mathbb{R}$ and $\xi \in 2^{j \sigma/2}\widehat{S}$. Then $\lvert t\rvert  \sim 2^j$ since $\lvert \xi\rvert  \ll 2^j$. Write 
\[ \xi = \xi_1 \gamma\left(\theta_{\tau}\right) + \xi_2 \gamma'\left(\theta_{\tau}\right) + \xi_3 (\gamma \times \gamma')\left(\theta_{\tau}\right), \]
where 
\[ \lvert \xi_1\rvert  \leq 2^{j\sigma/2} 2^{j(1-\sigma ) }, \quad \lvert \xi_2\rvert , \lvert \xi_3\rvert  \leq 2^{j\sigma/2} 2^{-j \sigma } . \]
Then $\left\lvert\langle \gamma(\theta), \gamma'(\theta_{\tau} ) \rangle \right\rvert \sim \lvert \theta-\theta_{\tau}\rvert $ for all $\lvert \theta-\theta_{\tau} \rvert  \leq r$ by $C^2$ Taylor approximation (for $r$ sufficiently small depending only on $\gamma$; using that $\gamma$ has unit speed). Hence 
\[ \left\lvert\langle  t\gamma(\theta)- \xi, \gamma'(\theta_{\tau} ) \rangle \right\rvert \geq \lvert t \rvert \left\lvert \left\langle \gamma(\theta), \gamma'(\theta_{\tau} ) \right\rangle \right\rvert - \lvert \xi_2\rvert \gtrsim 2^{j \sigma}, \]
 contradicting the assumption that $t\gamma(\theta) - \xi \in 1.1\tau$. Therefore, to prove \eqref{statphase} it remains to show that 
\begin{equation} \label{statphase2} \int_{\mathbb{R}^3 \setminus 2^{j\sigma/2}\widehat{S}} \left\lvert \widehat{\eta_S}(\xi )\right\rvert \, d\xi \leq C 2^{-j\sigma N }. \end{equation}
That \eqref{statphase2} implies \eqref{statphase} follows from Fubini's theorem; using that for fixed $\xi$, the set of $t$ for which $t \gamma(\theta) - \xi \in 1.1 \tau$ has length $\lesssim \left\lvert \theta-\theta_{\tau} \right\rvert^{-1}$. 

It is straightforward to check (e.g.~by \eqref{derivcondition} and integration by parts) that if $\xi \notin 2^{j\sigma/2}\widehat{S}$, then 
\[ \left\lvert \widehat{\eta_S}(\xi)\right\rvert \leq C_{N, \sigma} \mathcal{H}^3(S) \dist\left( \xi, \widehat{S} \right)^{-N}, \]
for any $N$. Substituting into \eqref{statphase2}, using that $\mathcal{H}^3(S) \mathcal{H}^3\left(\widehat{S}\right) \lesssim 1$, and summing the geometric series over the dyadic distance of $\xi$ from $\widehat{S}$, yields 
\[ \eqref{statphase2}  \lesssim C 2^{-j\sigma N }, \]
which finishes the proof.   \end{proof} 

The proof of the following lemma is very similar to \cite[Lemma~2.4]{harris24}, but the details are included for completeness.

\begin{lemma} \label{MTf} Let $j \geq 1$ and let $\tau \in \Lambda_j$. For any finite complex Borel measure $\mu$ and any $S \in \mathbb{S}_{\tau}$, 
\[ \left\lVert M_S \mu \right\rVert_{L^1(\mathbb{R}^3) } \leq 2^{3j \sigma } \lvert\mu \rvert(2S) + C_N 2^{-j \sigma N} \left\lvert\mu\right\rvert(\mathbb{R}^3), \]
for any positive integer $N$. \end{lemma}
\begin{proof} It will be assumed that $\mu$ is positive; the complex case is similar by replacing $\mu$ by $\lvert \mu\rvert $ in \eqref{prestarcircle} and \eqref{starcircle} below. By definition,
\begin{align}\notag \left\lVert M_S \mu \right\rVert_{L^1(\mathbb{R}^3) } &= \int \eta_S(x) \left\lvert \int \widecheck{\psi_{\tau}}(x-y) \, d\mu(y) \right\rvert \, dx \\
\label{prestarcircle} &\leq \int_{S} \int_{2S} \left\lvert \widecheck{\psi_{\tau}}(x-y) \right\rvert \, d\mu(y) \, dx \\
\label{starcircle} &\quad + \int_S \int_{\mathbb{R}^3 \setminus 2S} \left\lvert \widecheck{\psi_{\tau}}(x-y) \right\rvert \, d\mu(y) \, dx. \end{align}
By the Hausdorff-Young inequality and the property $\mathcal{H}^3(\tau) \mathcal{H}^3(S) = 2^{3j \sigma}$, the integral in \eqref{prestarcircle} satisfies 
\begin{equation} \label{star1} \int_{S} \int_{2S} \left\lvert \widecheck{\psi_{\tau}}(x-y) \right\rvert \, d\mu(y) \, dx  \leq 2^{3j \sigma } \mu(2S). \end{equation}
The integral in \eqref{starcircle} satisfies
\[ \int_S \int_{\mathbb{R}^3 \setminus 2S} \left\lvert \widecheck{\psi_{\tau}}(x-y) \right\rvert \, d\mu(y) \, dx \leq \mu(\mathbb{R}^3) \int_{\mathbb{R}^3 \setminus S_0} \left\lvert \widecheck{\psi_{\tau}} \right\rvert, \]  
where $S_0$ is the translate of the slab $S$ to the origin, parallel to $S$. Integration by parts (using \eqref{derivconditiontau}) gives that for any $k \geq 0$, for any $x \in \mathbb{R}^3 \setminus 2^k S_0$, 
\[ \left\lvert \widecheck{\psi_{\tau}}(x) \right\rvert \lesssim_N 2^{-kN-j \sigma N} \mathcal{H}^3(\tau). \]
Summing a geometric series over $k \geq 0$, using $\mathcal{H}^3(\tau) \mathcal{H}^3(S) \lesssim 2^{3j \sigma}$ again, gives
\begin{equation} \label{star2} \int_S \int_{\mathbb{R}^3 \setminus 2S} \left\lvert \widecheck{\psi_{\tau}}(x-y) \right\rvert \, d\mu(y) \, dx \leq C_N\mu(\mathbb{R}^3)2^{-j \sigma N}.  \end{equation}
Substituting \eqref{star1} and \eqref{star2} into \eqref{prestarcircle} and \eqref{starcircle} finishes the proof. \end{proof} 

\section{Proofs of the main results}

The first lemma of this section, Lemma~\ref{energylemma} below, is the main ingredient needed to prove Theorem~\ref{projmeasure}. It does not trivially imply Theorem~\ref{projmeasure}, but the actual proof of Theorem~\ref{projmeasure} will be very similar to the proof of Lemma~\ref{energylemma}, and will also use Lemma~\ref{energylemma} in a crucial way. Moreover, Lemma~\ref{energylemma} does imply \eqref{ganguthmaldagueinequality}, and minor modifications to the proof can yield \eqref{ganguthmaldagueinequality2}. In the statement of Lemma~\ref{energylemma}, the lower integral $\int_{*} f$ of a function $f$ is the supremum over the integrals of measurable functions $\phi$ with $0 \leq \phi \leq f$ (see \cite[p.~13]{mattila1995}) (which matches the standard integral when $f$ is measurable). The application of Lemma~\ref{energylemma} will be in the case where the integrand is measurable and the lower integral is just a standard integral, but using the lower integral avoids having to consider measurability issues in the statement.

\begin{lemma}  \label{energylemma} Let $\beta \in [0,1]$ and let $\alpha = 3-2\beta$. For any $\epsilon >0$, there exists $\delta>0$  and $C= C(\delta, \epsilon, \gamma)$ such that 
\begin{equation} \label{geometric} \int_{*}  (\rho_{\theta \sharp}\mu)\left( \bigcup_{D \in \mathbb{D}_{\theta}} D \right) \, d\lambda(\theta)  \leq C R^{-\delta} \mu(\mathbb{R}^3) c_{\beta}(\lambda)^{1/2} \lambda(\mathbb{R})^{1/2}, \end{equation}
for any $R \geq 1$,  for any Borel measure $\lambda$ supported in $I$ with $c_{\beta}(\lambda) < \infty$, for any Borel measure $\mu$ supported in $B_3(0,1)$ with $c_{\alpha}(\mu) \leq 1$, for any family of sets $\{ \mathbb{D}_{\theta} \}_{\theta \in I}$, where each $\mathbb{D}_{\theta}$ is a collection of radius $R^{-1}$ intervals in the span of $\gamma(\theta)$, of cardinality $\lvert\mathbb{D}_{\theta}\rvert \leq R^{1-\epsilon} \mu(\mathbb{R}^3)$. \end{lemma} 

\begin{proof}[Proof of Lemma~\ref{energylemma}] It may be assumed that $\gamma$ is restricted to an interval of diameter smaller than $r=r(\gamma)$ on which Lemma~\ref{IBP} holds. 

For reference, the ``size hierarchy'' of the various parameters, to be chosen below, will be 
\[ 0 < \delta \ll \sigma \ll \zeta \ll \epsilon - \frac{2\epsilon_0}{3} , \qquad 0< \frac{2\epsilon_0}{3} < \epsilon < \epsilon_0, \qquad \sigma \ll \delta_{\epsilon_0}. \]

Let $\phi_R$ be a non-negative bump function supported in $B_3(0, R^{-1})$ with $\int_{\mathbb{R}^3} \phi_R = 1$, defined by 
\[ \phi_R(x) = R^3 \phi(Rx), \]
for some fixed non-negative bump function $\phi$ supported in $B_3(0,1)$ with $\int_{\mathbb{R}^3} \phi = 1$. For any $\theta$ and any $D \in \mathbb{D}_{\theta}$, the 1-Lipschitz property of orthogonal projections implies that
\[ \left[ \rho_{\theta\sharp}(\mu \ast \phi_R) \right](2D) \geq (\rho_{\theta\sharp}\mu)(D), \]
where $2D$ is the interval with the same centre as $D$, but twice the radius. Moreover, 
\[ c_{\alpha}( \mu \ast \phi_R) \lesssim c_{\alpha}(\mu), \] 
so it suffices to prove \eqref{geometric} with $\mu \ast \phi_R$ in place of $\mu$. To simplify notation the new measure will not be relabelled, but it will be assumed throughout that $\mu$ is a non-negative Schwartz function, and that
\begin{equation} \label{schwartzdecay} \left\lvert \widehat{\mu}(\xi) \right\vert \leq C_{K} (R/ \lvert \xi\rvert)^{K}, \quad \xi \in \mathbb{R}^3, \end{equation}
for any positive integer $K$, where $C_{K}$ is a constant depending only on $K$ (i.e., $\widehat{\mu}$ is rapidly decaying outside $B(0,R)$). 

The conclusion of the lemma holds trivially for any $\epsilon >1$. Let $\epsilon_0$ be any positive real number which is strictly larger than the infimum over all positive $\epsilon$ for which the conclusion of the lemma is true (this infimum is well-defined by the preceding remark). It suffices to prove that the lemma holds for any $\epsilon > \frac{2\epsilon_0}{3}$ with $\epsilon < \epsilon_0$, so let such an $\epsilon$ be given. Let $\mu$, $\lambda$, and $\left\{ \mathbb{D}_{\theta} \right\}_{\theta \in I}$ be given. Let $R \geq 1$ and choose a non-negative integer $J$ such that $2^J \sim R^{\epsilon/1000}$. Let $\zeta >0$ be such that $\zeta \ll \epsilon - \frac{2\epsilon_0}{3}$. Choose $\sigma>0$ such that $\sigma \ll \min\{\delta_{\epsilon_0}, \zeta\}$, where $\delta_{\epsilon_0}$ is a $\delta$ corresponding to $\epsilon_0$ that satisfies \eqref{geometric}. 

Define the ``bad'' part of $\mu$ by
\begin{equation} \label{baddefn} \mu_b = \sum_{j \geq J} \sum_{\tau \in \Lambda_j} \sum_{S \in \mathbb{S}_{\tau,b} } M_S \mu, \end{equation}
where, for each $\tau \in \Lambda_j$, the set of ``bad'' slabs corresponding to $\tau$ is defined by 
\begin{equation} \label{badslabdefn} \mathbb{S}_{\tau, b} = \left\{ S \in \mathbb{S}_{\tau} : \mu(100S) \geq 2^{-j(1-\epsilon_0 )} \right\}, \end{equation}
where $100S$ is a slab with the same centre as $S$, but with side lengths scaled by a factor of 100. Define the ``good'' part of $\mu$ by 
\[ \mu_g = \mu-\mu_b. \] 
The rapid decay of $\widehat{\mu}$ outside $B(0,R)$ in \eqref{schwartzdecay} implies that the sum in \eqref{baddefn} converges (for example) in the Schwartz space $\mathcal{S}(\mathbb{R}^3)$. This implies that $\mu_b$ and $\mu_g$ are Schwartz functions, and are therefore finite complex measures. By the Cauchy-Schwarz inequality,
\begin{multline*} \int_{*} \left( \rho_{\theta \sharp} \mu \right)\left( \bigcup_{D \in \mathbb{D}_{\theta}} D \right) \, d\lambda(\theta)  \leq \int  \left\lVert \rho_{\theta \sharp} \mu_b \right\rVert_{L^1(\mathcal{H}^1) } \, d\lambda(\theta) \\ + \lambda(\mathbb{R})^{1/2}\sup_{\theta \in I} \mathcal{H}^1\left(  \bigcup_{D \in \mathbb{D}_{\theta}} D \right)^{1/2} \left(\int  \left\lVert \rho_{\theta \sharp} \mu_g \right\rVert_{L^2(\mathcal{H}^1) }^2 \, d\lambda(\theta) \right)^{1/2}. \end{multline*}
The contribution from the ``bad'' part will be bounded first. By the triangle inequality,
\begin{align} \notag \int  \left\lVert \rho_{\theta \sharp} \mu_b \right\rVert_{L^1(\mathcal{H}^1) } \, d\lambda(\theta)  &\leq \sum_{j \geq J}   \int \sum_{\tau \in \Lambda_j} \sum_{S \in \mathbb{S}_{\tau,b} }\left\lVert \rho_{\theta \sharp} M_S \mu \right\rVert_{L^1(\mathcal{H}^1) } \, d\lambda(\theta) \\
\label{mainpart} &= \sum_{j \geq J}  \int  \sum_{\substack{\tau \in \Lambda_j: \\ \left\lvert \theta_{\tau} - \theta \right\rvert \leq 2^{-j\left(1-\sigma\right)}}}  \sum_{S \in \mathbb{S}_{\tau,b} }\left\lVert \rho_{\theta \sharp} M_S \mu \right\rVert_{L^1(\mathcal{H}^1) } \, d\lambda(\theta)\\
\label{negligible} &\quad + \sum_{j \geq J}   \int  \sum_{\substack{\tau \in \Lambda_j: \\ \left\lvert \theta_{\tau} - \theta \right\rvert > 2^{-j\left(1-\sigma\right)}}} \sum_{S \in \mathbb{S}_{\tau,b} }\left\lVert \rho_{\theta \sharp} M_S \mu \right\rVert_{L^1(\mathcal{H}^1) } \, d\lambda(\theta). \end{align}
The sum in \eqref{negligible} is negligible by Lemma~\ref{IBP}, so it may be assumed that \eqref{mainpart} dominates. 
By the inequality
\[ \left\lVert \rho_{\theta \sharp} f \right\rVert_{L^1(\mathcal{H}^1)} \leq \left\lVert f \right\rVert_{L^1(\mathbb{R}^3)}, \]
followed by Lemma~\ref{MTf},
\begin{align*} \eqref{mainpart} &\leq  \sum_{j \geq J}  \int \sum_{\substack{\tau \in \Lambda_j: \\ \left\lvert \theta_{\tau} - \theta \right\rvert \leq 2^{-j\left(1-\sigma\right)}}}  \sum_{S \in \mathbb{S}_{\tau,b} }\left\lVert  M_S \mu \right\rVert_{L^1(\mathbb{R}^3) } \, d\lambda(\theta) \\
&\lesssim \sum_{j \geq J} 2^{3j \sigma} \int  \sum_{\substack{\tau \in \Lambda_j: \\ \left\lvert \theta_{\tau} - \theta \right\rvert \leq 2^{-j\left(1-\sigma\right)}}}  \sum_{S \in \mathbb{S}_{\tau,b} }\mu(2S) \, d\lambda(\theta), \end{align*}
where the negligible tail terms from Lemma~\ref{MTf} can be assumed to not dominate since the desired bound already follows in that case. The above satisfies 
\begin{equation} \label{nontail} \sum_{j \geq J} 2^{3j \sigma} \int  \sum_{\substack{\tau \in \Lambda_j: \\ \left\lvert \theta_{\tau} - \theta \right\rvert \leq 2^{-j\left(1-\sigma\right)}}}  \sum_{S \in \mathbb{S}_{\tau,b} }\mu(2S) \, d\lambda(\theta) \lesssim  \sum_{j \geq J} 2^{10j \sigma}\int \mu(B_j(\theta) ) \, d\lambda(\theta),  \end{equation}
where, for each $\theta \in I$ and each $j$, 
\[ B_j(\theta) =  \bigcup_{\tau \in \Lambda_j : \left\lvert \theta_{\tau} - \theta \right\rvert \leq 2^{-j\left(1-\sigma\right)}}  \bigcup_{S \in \mathbb{S}_{\tau,b} } 2S. \]
The inequality \eqref{nontail} used that for each $j$ and each $\theta \in I$, each of the slabs $2S$ in the union defining $B_j(\theta)$ intersects $\lesssim 2^{2j \sigma}$ of the others. Since the conclusion of the lemma holds for $\epsilon_0$, and since (by \eqref{badslabdefn}) the number of slabs in the union defining $B_j(\theta)$ is $\lesssim 2^{j \sigma } 2^{j(1-\epsilon_0) }$, it follows\footnote{Strictly speaking, the conclusion of the lemma is being applied at scale $R_j \sim 2^{j(1-\sigma)}$, whilst the cardinality bound is only $2^{j(1-\epsilon_0+\sigma)}$, so to obtain a cardinality bound of $R_j^{1-\epsilon_0}= 2^{j(1-\sigma)(1-\epsilon_0)}$ it is required to first take a maximally contributing subcollection of this size, with this step losing a factor of $2^{2j\sigma}$, and then the conclusion of the lemma is applied to this subcollection.} that for each $j \geq J$,
\begin{multline*}  \int \mu(B_j(\theta) ) \, d\lambda(\theta) \leq \int \left( \rho_{\theta\sharp} \mu \right) \left(\bigcup_{\tau \in \Lambda_j : \left\lvert \theta_{\tau} - \theta \right\rvert \leq 2^{-j\left(1-\sigma\right)}} \bigcup_{ S \in \mathbb{S}_{\tau,b}} 2 \rho_{\theta}(S) \right) \, d\lambda(\theta)\\
 \lesssim 2^{j \left(-\delta_{\epsilon_0}/2 + 10 \sigma \right)} \mu(\mathbb{R}^3) \lambda(\mathbb{R})^{1/2} c_{\beta}(\lambda)^{1/2} . \end{multline*}
The set $B_j(\theta)$ is piecewise constant in $\theta$ over a partition of $I$ into Borel sets, so the integral above equals the lower integral as in the statement of the lemma. Since $\sigma \ll \delta_{\epsilon_0}$, summing the above inequality over $j$ yields
\begin{multline*} \eqref{mainpart} \lesssim2^{-(J \delta_{\epsilon_0})/100}  \mu(\mathbb{R}^3)\lambda(\mathbb{R})^{1/2} c_{\beta}(\lambda)^{1/2} \\ \lesssim R^{-(\epsilon \delta_{\epsilon_0})/10^5} \mu(\mathbb{R}^3)\lambda(\mathbb{R})^{1/2} c_{\beta}(\lambda)^{1/2}. \end{multline*}

It remains to bound the contribution from $\mu_g$.  By the assumptions in the lemma, 
\[  \sup_{\theta \in I} \mathcal{H}^1\left(  \bigcup_{D \in \mathbb{D}_{\theta}} D \right) \lesssim  R^{-\epsilon} \mu(\mathbb{R}^3). \]  
Since $\zeta \ll \epsilon - (2\epsilon_0)/3$, it suffices to prove that
\begin{equation} \label{sufficient1} \int \left\lVert \rho_{\theta \sharp} \mu_g \right\rVert_{L^2(\mathcal{H}^1) }^2 \, d\lambda(\theta) \lesssim \max\left\{R^{2\epsilon_0/3 + 1000 \zeta}, R^{\epsilon/2} \right\} c_{\beta}(\lambda)\mu(\mathbb{R}^3). \end{equation}
By Plancherel's theorem in 1 dimension,
\begin{equation} \label{plancherel1d} \int  \left\lVert \rho_{\theta \sharp} \mu_g \right\rVert_{L^2(\mathcal{H}^1) }^2\, d\lambda(\theta) = \int \int_{\mathbb{R}} \left\lvert \widehat{ \mu_g} ( t \gamma(\theta) ) \right\rvert^2 \, dt \, d\lambda(\theta). \end{equation} 
By symmetry and by summing a geometric series, using $\sigma \ll \zeta$, to prove \eqref{sufficient1} it will suffice to show that
\begin{equation} \label{sufficient2} \int \int_{2^{j-1}}^{2^j}\left\lvert \widehat{\mu_g}\left( t \gamma(\theta) \right) \right\rvert^2 \, dt \, d\lambda(\theta) \lesssim 2^{j\left( \frac{2\epsilon_0}{3} + 100\zeta +  O\left( \sigma \right) \right) } c_{\beta}(\lambda) \mu(\mathbb{R}^3),   \end{equation}
for any $j \geq 2J$. The contribution from the small frequencies ($j < 2J$) is controlled by the definition $2^{J} \sim R^{\epsilon/1000}$ of $J$ (this is the reason for the $R^{\epsilon/2}$ term in \eqref{sufficient1}), and the contribution from the large frequencies ($2^j \geq R^{1+\sigma}$) is controlled by the rapid decay of $\widehat{\mu}$ outside $B(0,R)$ (see \eqref{schwartzdecay}) instead of \eqref{sufficient2}. For each $\tau \in \Lambda$, define the set of ``good'' slabs corresponding to $\tau$ by
\[ \mathbb{S}_{\tau,g} = \mathbb{S}_{\tau} \setminus \mathbb{S}_{\tau, b}. \]
Fix a $j \geq 2J$ as in \eqref{sufficient2}. Then, apart from a negligible error term which can be assumed to not dominate,
\begin{multline} \label{pause29} \int \int_{2^{j-1}}^{2^j}\left\lvert \widehat{\mu_g}\left( t \gamma(\theta) \right) \right\rvert^2 \, dt \, d\lambda(\theta) \\
\lesssim \int  \int_{2^{j-1}}^{2^j}\left\lvert \sum_{\tau \in \bigcup_{\left\lvert j' -j\right\rvert \leq 2} \Lambda_{j'}} \sum_{S \in \mathbb{S}_{\tau,g} } \widehat{M_S\mu}\left( t \gamma(\theta) \right) \right\rvert^2 \, dt \, d\lambda(\theta). \end{multline}
Let $\Phi_j$ be a boundedly overlapping cover of the $\sim 1$ neighbourhood of the cone at distance $\sim 2^j$ from the origin by standard boxes (or planks) $\phi$ of dimensions $\sim 1 \times 2^{j/2} \times 2^j$, separated by a distance $\sim 2^{j}$ from the origin and tangent to the cone.  Let $\{ \psi_{\phi} \}_{\phi \in \Phi_j}$ be such that each $\psi_{\phi}$ is a bump function supported in $\phi$, and such that each $\Lambda_{j'}$ with $\lvert j-j'\rvert  \leq 2$ can be partitioned as $\Lambda_{j'} = \bigcup_{\phi \in \Phi_j} \Lambda_{j', \phi}$, where each $\tau \in \Lambda_{j', \phi}$ satisfies $1.1\tau \subseteq \phi$ and has the property that $\psi_{\phi} = 1$ on $1.1 \tau$. Then
\[  \sum_{\tau \in \bigcup_{\left\lvert j' -j\right\rvert \leq 2} \Lambda_{j'}} \sum_{S \in \mathbb{S}_{\tau,g} } \widehat{M_S\mu} = \sum_{\phi \in \Phi_j} \sum_{\tau \in \bigcup_{\left\lvert j' -j\right\rvert \leq 2} \Lambda_{j', \phi}} \sum_{S \in \mathbb{S}_{\tau,g} } \widehat{M_S\mu}. \]
Since the $\phi$'s and the $\tau$'s are boundedly overlapping, 
\begin{equation} \label{intermezzo} \eqref{pause29} \lesssim \sum_{\phi \in \Phi_j} \sum_{\tau \in \bigcup_{\left\lvert j' -j\right\rvert \leq 2} \Lambda_{j', \phi}} \int \int_{\mathbb{R}} \left\lvert  \sum_{S \in \mathbb{S}_{\tau,g} } \widehat{M_S\mu}\left( t \gamma(\theta) \right) \right\rvert^2 \, dt \, d\lambda(\theta), \end{equation}
where again the negligible error terms can be assumed to not dominate\footnote{From now on, such negligible error terms will mostly be ignored without comment when they can be assumed to not dominate.}, and the integration domain is now over $\mathbb{R}$ in order to apply Plancherel. If $\theta$ is fixed and then the application of the 1-dimensional Plancherel theorem (see \eqref{plancherel1d}) is reversed, then the sets $\rho_{\theta}(S)$ with $S \in \bigcup_{\tau \in \bigcup_{\lvert j'-j\rvert  \leq 2} \Lambda_{j'} : \left\lvert \theta_{\tau}-\theta\right\rvert \lesssim 2^{-j\left(1-\sigma\right)}} \mathbb{S}_{\tau}$  are $\lesssim 2^{2j\sigma}$ overlapping, and these are the only $S$ making a non-negligible contribution to the sum inside the integral in \eqref{intermezzo} (e.g.~by Lemma~\ref{IBP}). It follows that 
\begin{multline} \label{intermezzo3} \eqref{intermezzo} \lesssim \sum_{\phi \in \Phi_j} \sum_{\tau \in \bigcup_{\left\lvert j' -j\right\rvert \leq 2}  \Lambda_{j', \phi}}  \sum_{S \in \mathbb{S}_{\tau,g} } \int \int_{2^{j-10}}^{2^{j+10}}\left\lvert  \widehat{M_S\mu}\left( t \gamma(\theta) \right) \right\rvert^2 \, dt \, d\lambda(\theta). \end{multline}
 For each $\phi \in \Phi_j$, let $\mathbb{T}_{\phi}$ be a boundedly overlapping cover of $\mathbb{R}^3$ by planks of dimensions $2^{j\sigma} \times 2^{j \left( \sigma - 1/2 \right) } \times 2^{j\left( \sigma-1 \right)}$ dual to $\phi$, and let $\{\eta_T\}_{T \in \mathbb{T}_{\phi}}$ be a corresponding smooth partition of unity subordinate to this cover. Given $T$, write $\phi(T)$ for the $\phi$ such that $T \in \mathbb{T}_{\phi}$. Similarly, given $S$, write $\tau(S)$ for the $\tau$ such that $S \in \mathbb{S}_{\tau}$. For each $S$ and $T$, let 
\[ M_{S,T} \mu = \eta_S \left( \left[ \eta_T\left( \mu \ast \widecheck{\psi_{\phi(T) } } \right) \right] \ast \widecheck{\psi_{\tau(S) } } \right). \]
This can be written as 
\[ M_{S,T} \mu = M_S\left( M_T \mu\right), \]
where 
\[ M_T \nu := \eta_T\left( \nu \ast \widecheck{ \psi_{\phi(T) } } \right). \]
For any $S$, if $\tau(S) \in \Lambda_{j', \phi}$, then $\psi_{\phi} =1$ on $1.1 \tau$, and therefore
\[ M_S \mu =  \sum_{T \in \mathbb{T}_{\phi} } M_{S,T} \mu. \]
For each $\phi$ and dyadic number $\kappa$, let 
\[ \mathbb{T}_{\phi,\kappa} = \left\{ T \in \mathbb{T}_{\phi} : \kappa \leq \mu(4T) < 2\kappa \right\}. \]
 By Cauchy-Schwarz, and since there are $\lesssim \log(2^j)$ many dyadic values $\kappa$ contributing substantially\footnote{Strictly speaking, it is necessary to apply the complex case of Lemma~\ref{MTf}, and then Lemma~2.3 from \cite{harris24} (the analogue of Lemma~\ref{MTf} for the $M_T$), to rule out extremely small dyadic values.} to \eqref{intermezzo3}, there is a fixed dyadic number $\kappa = \kappa(j)$ such that
\begin{multline} \label{firstpigeon} \eqref{intermezzo3} \lesssim \\
\log(2^j)  \sum_{\phi \in \Phi_j} \sum_{\tau \in \bigcup_{\left\lvert j' -j\right\rvert \leq 2} \Lambda_{j', \phi}}  \sum_{S \in \mathbb{S}_{\tau,g} } \int \int_{2^{j-10}}^{2^{j+10}}\left\lvert \sum_{T \in \mathbb{T}_{\phi, \kappa}}  \widehat{M_{S,T} \mu}\left( t \gamma(\theta) \right) \right\rvert^2 \, dt \, d\lambda(\theta). \end{multline}

By further dyadic pigeonholing, there is a dyadic number $N = N(j)$ such that for every $\phi \in \Phi_j$ and $\tau  \in \bigcup_{\left\lvert j' -j\right\rvert \leq 2} \Lambda_{j', \phi}$, there is a subset $\mathbb{S}_{\tau,g,N} \subseteq \mathbb{S}_{\tau,g}$ such that every $S \in \mathbb{S}_{\tau,g,N}$ is such that $2S$ intersects between $N$ and $2N$ planks $T \in \mathbb{T}_{\phi, \kappa}$, and such that 
\begin{multline} \label{secondpigeon} \eqref{firstpigeon}
\lesssim \\
\log(2^j)^2  \sum_{\phi \in \Phi_j} \sum_{\tau \in \bigcup_{\left\lvert j' -j\right\rvert \leq 2} \Lambda_{j', \phi}}  \sum_{S \in \mathbb{S}_{\tau,g,N} } \int \int_{2^{j-10}}^{2^{j+10}}\left\lvert \sum_{T \in \mathbb{T}_{\phi, \kappa}}  \widehat{M_{S,T} \mu}\left( t \gamma(\theta) \right) \right\rvert^2 \, dt \, d\lambda(\theta). \end{multline}
In \eqref{secondpigeon}, if $T \in \mathbb{T}_{\phi, \kappa}$ does not intersect any set $2S$ with 
\[ S \in \bigcup_{\tau \in \bigcup_{\left\lvert j' -j\right\rvert \leq 2} \Lambda_{j', \phi}}\mathbb{S}_{\tau,g,N}, \]
 then, by the complex case of Lemma~\ref{MTf}, it makes negligible contribution to \eqref{secondpigeon}, so after removing some of the $T$ from each $\mathbb{T}_{\phi, \kappa}$, it may be assumed that for every $\phi \in \Phi_j$, every $T \in \mathbb{T}_{\phi, \kappa}$ intersects some set $2S$ with $S \in \mathbb{S}_{\tau,g,N}$ for some $\tau \in \bigcup_{\left\lvert j' -j\right\rvert \leq 2} \Lambda_{j', \phi}$. This refinement does not affect the pigeonholed property that for every $\phi \in \Phi_j$, every $S \in \mathbb{S}_{\tau,g,N}$ with $\tau \in \bigcup_{\left\lvert j' -j\right\rvert \leq 2} \Lambda_{j', \phi}$ is such that $2S$ intersects between $N$ and $2N$ many planks $T \in \mathbb{T}_{\phi, \kappa}$. This implies the following:
\begin{equation} \label{goodplank1} \mu(4T) \lesssim 2^{-j(1-\epsilon_0)} N^{-1}, \qquad  \forall T \in \bigcup_{\phi \in \Phi_j } \mathbb{T}_{\phi,\kappa}. \end{equation}
To verify \eqref{goodplank1}, let $T \in \mathbb{T}_{\phi,\kappa}$ for some $\phi \in \Phi_j$. Then (by the refinement) there exists $S \in \mathbb{S}_{\tau,g,N}$ with $\tau \in \bigcup_{\left\lvert j' -j\right\rvert \leq 2} \Lambda_{j', \phi}$ and $2S \cap T \neq \emptyset$, so by definition of $\mathbb{S}_{\tau,g,N}$ there are $\gtrsim N$ many (boundedly overlapping) $T' \in \mathbb{T}_{\phi, \kappa}$ with $4T' \subseteq 100S$ and $\mu(4T') \sim \mu(4T)$. This relies on the geometric property that 
\[ T \cap 2S \neq \emptyset  \Rightarrow 4T \subseteq 100S \qquad \text{ for } \tau(S) \in \Lambda_{j', \phi(T) }. \]
Since the $T' \in \mathbb{T}_{\phi, \kappa}$ are boundedly overlapping, this yields
\[ N \mu(4T) \lesssim \mu(100S). \]
The inequality \eqref{goodplank1} then follows from the above together with the defining property $\mu(100S) < 2^{-j(1-\epsilon_0 )}$ of the ``good'' slabs from \eqref{badslabdefn}.

For each fixed $\phi \in \Phi_j$ (with $j$ still fixed), the innermost double sum from \eqref{secondpigeon}:
\begin{equation} \label{aboveintegral}  \sum_{\tau \in \bigcup_{\left\lvert j' -j\right\rvert \leq 2} \Lambda_{j', \phi}}  \sum_{S \in \mathbb{S}_{\tau,g,N} } \int \int_{2^{j-10}}^{2^{j+10}}\left\lvert \sum_{T \in \mathbb{T}_{\phi, \kappa}}  \widehat{M_{S,T} \mu}\left( t \gamma(\theta) \right) \right\rvert^2 \, dt \, d\lambda(\theta), \end{equation}
will be bounded in two different ways. 

First, for each $\tau$ and $S$, the double integral in \eqref{aboveintegral} is bounded by
\begin{multline} \label{omfgthisisannoying}   \int_{|\theta-\theta_{\tau}| \leq 2^{-j(1-\sigma)}} \int_{2^{j-10}}^{2^{j+10}}\left\lvert \sum_{T \in \mathbb{T}_{\phi, \kappa}}  \widehat{M_{S,T} \mu}\left( t \gamma(\theta) \right) \right\rvert^2 \, dt \, d\lambda(\theta) \\
+ \int_{|\theta-\theta_{\tau}| > 2^{-j(1-\sigma)}} \int_{2^{j-10}}^{2^{j+10}}\left\lvert \sum_{T \in \mathbb{T}_{\phi, \kappa}}  \widehat{M_{S,T} \mu}\left( t \gamma(\theta) \right) \right\rvert^2 \, dt \, d\lambda(\theta). \end{multline} 
By applying the Fourier slice theorem (in one dimension), then the Hausdorff-Young inequality (in one dimension) to bound the $L^{\infty}$ norm of the Fourier transform by the $L^1$ norm, followed by Lemma~\ref{IBP}, the second term in \eqref{omfgthisisannoying} is negligible and may therefore be assumed not to dominate below. For the first term in \eqref{omfgthisisannoying}, by Hausdorff-Young and Cauchy-Schwarz, 
\[ \left\lVert \sum_{T \in \mathbb{T}_{\phi, \kappa}}  \widehat{M_{S,T} \mu} \right\rVert_{\infty} \leq \left\lVert \sum_{T \in \mathbb{T}_{\phi, \kappa}}  M_{S,T} \mu\right\rVert_{L^1(\mathbb{R}^3)} \lesssim 2^{j(-1/2+3\sigma)} \left\lVert \sum_{T \in \mathbb{T}_{\phi, \kappa}}  M_{S,T} \mu\right\rVert_{L^2(\mathbb{R}^3)}. \]
Squaring this, then substituting into the first term in \eqref{omfgthisisannoying} and applying the Frostman condition on $\lambda$, gives 
\begin{equation} \label{r3integral} \eqref{aboveintegral} \lesssim c_{\beta}(\lambda)2^{-j\left(\beta-100\sigma\right)} \sum_{\tau \in \bigcup_{\left\lvert j' -j\right\rvert \leq 2} \Lambda_{j', \phi}}  \sum_{S \in \mathbb{S}_{\tau,g,N} }  \int_{\mathbb{R}^3} \left\lvert \sum_{T \in \mathbb{T}_{\phi, \kappa}}  \widehat{M_{S,T} \mu} \right\rvert^2 \, d\xi. \end{equation} 
By writing
\[ \sum_{T \in \mathbb{T}_{\phi, \kappa} } M_{S,T}\mu = M_S \left( \sum_{T \in \mathbb{T}_{\phi, \kappa} } M_T\mu\right), \]
and then applying Plancherel's theorem in $\mathbb{R}^3$ to \eqref{r3integral}, summing over $S$, then applying Plancherel again and summing over $\tau$, then the $T$,
\begin{equation} \label{uncertainty1} \eqref{aboveintegral} \lesssim  c_{\beta}(\lambda)2^{-j\left(\beta-100\sigma\right)}  \sum_{T \in \mathbb{T}_{\phi, \kappa}} \int_{\mathbb{R}^3} \left\lvert  \widehat{M_{T} \mu} \right\rvert^2 \, d\xi. \end{equation}
 This proves the first bound of \eqref{aboveintegral}.

For the second (more difficult) bound of \eqref{aboveintegral}, the idea is that if there were only one plank $T$ in the sum defining the integrand of \eqref{aboveintegral}, then, since $M_T\mu$ is just a wave packet, the function 
\[ \sum_{S: T \subseteq S} \left\lvert \widehat{M_{S,T}\mu}( t\gamma(\theta) )\right\rvert^2 \]
 can be treated as constant as $\theta$ varies over a larger arc of length $2^{-j/2}$ rather than $2^{-j}$, resulting in less loss obtained by removing the fractal measure $\lambda$. This would replace the factor $2^{-j\beta}$ in \eqref{uncertainty1} by the smaller factor $2^{-j\left( \frac{1+\beta}{2}\right)}$. Since there are $\sim N$ planks intersecting a slab instead of just one, forcing this heuristic to work results in a loss by a factor of $N$, so a suitable weighted geometric mean of this bound with the preceding bound \eqref{uncertainty1} will be used, in such a way that the loss of $N$ cancels with a later gain of $N^{-1}$ obtained by applying \eqref{goodplank1}. 

The preceding heuristic will now be made precise. By the Cauchy-Schwarz inequality, 
\[ \eqref{aboveintegral} \lesssim \\
 \sum_{\tau \in \bigcup_{\left\lvert j' -j\right\rvert \leq 2} \Lambda_{j', \phi}}  \sum_{S \in \mathbb{S}_{\tau,g,N} } \sum_{T \in \mathbb{T}_{\phi, \kappa}} N  \int \int_{2^{j-10}}^{2^{j+10}}\left\lvert  \widehat{M_{S,T} \mu}\left( t \gamma(\theta) \right) \right\rvert^2 \, dt \, d\lambda(\theta), \]
where the contribution from those $T \in \mathbb{T}_{\phi,\kappa}$ with $T \cap 2S =\emptyset$ was ignored using Lemma~\ref{MTf} (after first applying Hausdorff-Young to bound the $L^{\infty}$ norm of the Fourier transform by the $L^1$ norm).  The sum can be relaxed to a sum over all $S \in \mathbb{S}_{\tau}$, since the properties of the ``good slabs'' are already contained in the planks $T$, and the sums over $\tau$ and $S$ in \eqref{starfix} can be moved inside the integral:
\[ \eqref{aboveintegral} \lesssim \\
  \sum_{T \in \mathbb{T}_{\phi, \kappa}} N   \int \int_{2^{j-10}}^{2^{j+10}} \sum_{\tau \in \bigcup_{\left\lvert j' -j\right\rvert \leq 2} \Lambda_{j', \phi}}  \sum_{S \in \mathbb{S}_{\tau} } \left\lvert  \widehat{M_{S,T} \mu}\left( t \gamma(\theta) \right) \right\rvert^2 \, dt \, d\lambda(\theta). \]		
Let $\theta_{\phi}$ be such that $\gamma(\theta_{\phi})$ is parallel to the long direction of the plank $\phi$. Then $\left\lvert \theta_{\tau} - \theta_{\phi} \right\rvert \lesssim 2^{-j/2}$ for any $\tau \in \bigcup_{\left\lvert j' -j\right\rvert \leq 2} \Lambda_{j', \phi}$. Hence
\begin{multline} \label{starfixpre2} \eqref{aboveintegral} \leq \\
  \sum_{T \in \mathbb{T}_{\phi, \kappa}} N   \int_{\left\lvert \theta_{\phi} - \theta \right\rvert \lesssim  2^{-j/2}} \int_{2^{j-10}}^{2^{j+10}} \sum_{\tau \in \bigcup_{\left\lvert j' -j\right\rvert \leq 2} \Lambda_{j', \phi}}  \sum_{S \in \mathbb{S}_{\tau} } \left\lvert  \widehat{M_{S,T} \mu}\left( t \gamma(\theta) \right) \right\rvert^2 \, dt \, d\lambda(\theta) \\
+   \sum_{T \in \mathbb{T}_{\phi, \kappa}} N  \sum_{\tau \in \bigcup_{\left\lvert j' -j\right\rvert \leq 2} \Lambda_{j', \phi}}  \sum_{S \in \mathbb{S}_{\tau} }  \int_{\left\lvert \theta_{\tau} - \theta \right\rvert > 2^{-j(1-\sigma) }} \int_{2^{j-10}}^{2^{j+10}} \left\lvert  \widehat{M_{S,T} \mu}\left( t \gamma(\theta) \right) \right\rvert^2 \, dt \, d\lambda(\theta). \end{multline}
By the Fourier slice theorem, followed by 1-dimensional Hausdorff-Young inequality to bound the $L^{\infty}$ norm by an $L^1$ norm, followed by the complex case of Lemma~\ref{IBP} (and then \cite[Lemma~2.3]{harris24}), the second term in \eqref{starfixpre2} is negligible, so it may be assumed that 
\begin{multline*}  \eqref{aboveintegral} \lesssim \\
  \sum_{T \in \mathbb{T}_{\phi, \kappa}} N   \int_{\left\lvert \theta_{\phi} - \theta \right\rvert \lesssim  2^{-j/2}} \int_{2^{j-10}}^{2^{j+10}} \sum_{\tau \in \bigcup_{\left\lvert j' -j\right\rvert \leq 2} \Lambda_{j', \phi}}  \sum_{S \in \mathbb{S}_{\tau} } \left\lvert  \widehat{M_{S,T} \mu}\left( t \gamma(\theta) \right) \right\rvert^2 \, dt \, d\lambda(\theta). \end{multline*}
By Lemma~\ref{IBP} (together with the Fourier slice theorem and Hausdorff-Young), this becomes 
\begin{multline*}  \eqref{aboveintegral} \lesssim \\
  \sum_{T \in \mathbb{T}_{\phi, \kappa}} N   \int_{\left\lvert \theta_{\phi} - \theta \right\rvert \lesssim  2^{-j/2}} \int_{2^{j-10}}^{2^{j+10}} \sum_{\substack{\tau \in \bigcup_{\left\lvert j' -j\right\rvert \leq 2} \Lambda_{j', \phi} \\
		\left\lvert \theta_{\tau} - \theta\right\rvert \leq 2^{-j(1-\sigma)}}}  \sum_{S \in \mathbb{S}_{\tau} } \left\lvert  \widehat{M_{S,T} \mu}\left( t \gamma(\theta) \right) \right\rvert^2 \, dt \, d\lambda(\theta), \end{multline*}
		apart from a negligible error term which may be assumed to not dominate. For $S$ and $T$ in the above sum, if $S \cap T \neq \emptyset$, then $M_T\mu \ast \widecheck{\psi_{\tau}}$ is negligible outside $100S$. Therefore, the above can be further refined to
		\begin{multline*}  \eqref{aboveintegral} \lesssim \sum_{T \in \mathbb{T}_{\phi, \kappa}} \\ N   \int_{\left\lvert \theta_{\phi} - \theta \right\rvert \lesssim  2^{-j/2}} \int_{2^{j-10}}^{2^{j+10}} \sum_{\substack{\tau \in \bigcup_{\left\lvert j' -j\right\rvert \leq 2} \Lambda_{j', \phi} \\
		\left\lvert \theta_{\tau} - \theta\right\rvert \leq 2^{-j(1-\sigma)}}}  \sum_{\substack{ S \in \mathbb{S}_{\tau} \\ 100S \cap T \neq \emptyset }} \left\lvert  \widehat{M_{S,T} \mu}\left( t \gamma(\theta) \right) \right\rvert^2 \, dt \, d\lambda(\theta), \end{multline*}
		By Bernstein's inequality (see e.g.~\cite[Proposition~5.4]{wolff}), 
		\[ \left\lVert  \widehat{M_{S,T} \mu} \right\rVert_{\infty} \lesssim \mathcal{H}^3(S)^{1/2} \left\lVert \widehat{M_{S,T} \mu} \right\rVert_2, \]
		so the above becomes
		\begin{multline} \label{starfixpre5} \eqref{aboveintegral} \lesssim \sum_{T \in \mathbb{T}_{\phi, \kappa}} \\ N   \int_{\left\lvert \theta_{\phi} - \theta \right\rvert \lesssim  2^{-j/2}} \int_{2^{j-10}}^{2^{j+10}} \sum_{\substack{\tau \in \bigcup_{\left\lvert j' -j\right\rvert \leq 2} \Lambda_{j', \phi} \\
		\left\lvert \theta_{\tau} - \theta\right\rvert \leq 2^{-j(1-\sigma)}}} \sum_{\substack{ S \in \mathbb{S}_{\tau} \\ 100S \cap T \neq \emptyset }}\mathcal{H}^3(S)\left\lVert  \widehat{M_{S,T} \mu} \right\rVert_2^2 \, dt \, d\lambda(\theta). \end{multline}
But by Plancherel and Hausdorff-Young,
\[ \left\lVert  \widehat{M_{S,T} \mu} \right\rVert_2^2 \lesssim \mathcal{H}^3(S) \left\lVert M_T \mu \ast \widecheck{\psi_{\tau}} \right\rVert_{\infty}^2 
\lesssim  \mathcal{H}^3(S) \left\lVert \widehat{M_T \mu} \psi_{\tau} \right\rVert_{1}^2 \lesssim  \mathcal{H}^3(S) \mathcal{H}^3(\tau)^2 \left\lVert \widehat{M_T \mu} \right\rVert_{\infty}^2.\]
Substituting into \eqref{starfixpre5}, using that for each fixed $\theta$ and $T$, the number of $\tau$ and $S$ occurring in the sum is $\lesssim 2^{j\sigma}$, gives
		\begin{equation} \label{starfix} \eqref{aboveintegral} \lesssim 2^{10j\sigma} \sum_{T \in \mathbb{T}_{\phi, \kappa}} N   \int_{\left\lvert \theta_{\phi} - \theta \right\rvert \lesssim  2^{-j/2}}  
		\int_{2^{j-10}}^{2^{j+10}}  \left\lVert \widehat{M_T \mu}  \right\rVert_{L^{\infty}}^2 \, dt \, d\lambda(\theta). \end{equation}
For any $T$, $\left\lVert \widehat{M_{T}\mu } \right\rVert_{\infty} \lesssim \mathcal{H}^3(T)^{1/2} \left\lVert \widehat{M_T \mu} \right\rVert_2$ by Bernstein's inequality again (\cite[Proposition~5.4]{wolff}). Substituting into \eqref{starfix} and applying the Frostman condition on $\lambda$ yields
\begin{equation} \label{uncertainty2} \eqref{aboveintegral} \lesssim c_{\beta}(\lambda) N 2^{-j\left(  \frac{1+\beta}{2} -100\sigma\right)} \sum_{T \in \mathbb{T}_{\phi, \kappa}}   \int_{\mathbb{R}^3} \left\lvert  \widehat{M_{T} \mu} \right\rvert^2 \, d\xi. \end{equation}
This is the second bound on \eqref{aboveintegral}. 

By taking a weighted geometric mean of \eqref{uncertainty1} and \eqref{uncertainty2}, 
\begin{multline} \label{uncertainty3} \eqref{aboveintegral} \lesssim \eqref{uncertainty1}^{1/3} \eqref{uncertainty2}^{2/3} \\
\lesssim 
 c_{\beta}(\lambda) N^{2/3} 2^{-j\left( \frac{1+2\beta}{3} -100\sigma \right)} \sum_{T \in \mathbb{T}_{\phi, \kappa}}   \int_{\mathbb{R}^3} \left\lvert  \widehat{M_{T} \mu} \right\rvert^2 \, d\xi. \end{multline}
By substituting \eqref{uncertainty3} into \eqref{secondpigeon}, then \eqref{firstpigeon}, then \eqref{intermezzo3}, then \eqref{intermezzo}, and then \eqref{pause29}, to prove \eqref{sufficient2} it suffices to show that for any $j \geq 2J$, with $N= N(j)$,
\begin{equation} \label{enough}  \sum_{\phi \in \Phi_{j} } \sum_{T \in \mathbb{T}_{\phi, \kappa}} \int_{\mathbb{R}^3} \left\lvert M_{T} \mu\right\rvert^2 \, dx \lesssim N^{-2/3} 2^{j\left( \frac{1+2\beta}{3} + \frac{2 \epsilon_0}{3} + 200 \zeta \right) } \mu(\mathbb{R}^3). \end{equation}
The rest of the proof will verify \eqref{enough}. Since the small cap wave packet decomposition has been converted to a standard wave packet decomposition, the application of refined decoupling is nearly identical to the proof of (2.18) in \cite{harris24}, but the details will be included for completeness, and to show the key gain over \cite{harris24} coming from the extra $N^{-1}$ factor in \eqref{goodplank1}. Let
\[ \mathbb{T}_{j,g} = \bigcup_{\phi \in \Phi_j} \mathbb{T}_{\phi,\kappa}. \]
For each $T \in \mathbb{T}_{j,g}$, let 
\begin{equation} \label{fTdefn}  f_T :=  \left[\eta_{T} M_{T} \mu \right] \ast \widecheck{\psi_{\phi(T)}}. \end{equation}
For each $T$, by unpacking the definition $M_T\mu = \eta_T \left( \mu \ast \widecheck{\psi_{\phi(T)}} \right)$ of one of the $M_T\mu$'s in \eqref{enough}, the left-hand side of \eqref{enough} is equal to
\begin{equation} \label{cauchymeasure}  \int \sum_{T \in \mathbb{T}_{j,g}} f_T \, d\mu. \end{equation}
By the Cauchy-Schwarz inequality with respect to $\mu$, 
\begin{equation} \label{mucauchyschwarz} \eqref{cauchymeasure} \leq \left(\int \left\lvert \sum_{T \in \mathbb{T}_{j,g}} f_T \right\rvert^2 \, d\mu \right)^{1/2}\cdot \mu(\mathbb{R}^3)^{1/2}. \end{equation}
By the uncertainty principle (since each $f_T$ is Fourier supported in a ball of radius $\sim 2^j$),
\[  \left(\int  \left\lvert \sum_{T \in \mathbb{T}_{j,g}} f_T \right\rvert^2 \, d\mu \right)^{1/2} \lesssim  \left(\int  \left\lvert \sum_{T \in \mathbb{T}_{j,g}} f_T \right\rvert^2 \, d\mu_j\right)^{1/2}, \]
where $\mu_j = \mu \ast \phi_j$ and $\phi_j(x) = \frac{2^{3j}}{1+\left\lvert 2^j x\right\rvert^{N_1}}$ , where $N _1\sim 1000/ \sigma^2$. Here $\mu_j$ should be thought of as the smoothed out version of $\mu$ at scale $2^{-j}$ and constant on balls of radius $2^{-j}$. By dyadic pigeonholing (with $2 \leq p\leq 6$ to be chosen), there exists a collection $\mathbb{W} \subseteq \mathbb{T}_{j,g}$ with $\left\lVert f_T\right\rVert_p$ constant over $T \in \mathbb{W}$ up to a factor of 2, and a union $Y$ of disjoint $2^{-j}$-balls $Q$ such that each $Q$ intersects $\sim M$ planks $2T$ with $T \in \mathbb{W}$ for some dyadic number $M$, and such that (ignoring negligible error terms which can be assumed to not dominate)
\[  \left(\int \left\lvert \sum_{T \in \mathbb{T}_{j,g}} f_T \right\rvert^2 \, d\mu_j\right)^{1/2} \lesssim j^{10} \left(\int_{Y} \left\lvert \sum_{T \in \mathbb{W}} f_T \right\rvert^2 \, d\mu_j\right)^{1/2}. \]
By Hölder's inequality with respect to Lebesgue measure,
\begin{equation} \label{pause} \left(\int_{Y} \left\lvert \sum_{T \in \mathbb{W}} f_T \right\rvert^2 \, d\mu_j\right)^{1/2} \leq \left\lVert  \sum_{T \in \mathbb{W}} f_T \right\rVert_{L^p(Y)} \left( \int_Y \mu_j^{\frac{p}{p-2} }\right)^{\frac{p-2}{2p} }, \end{equation}
By the dimension condition $c_{\alpha}(\mu) \leq 1$ on $\mu$, the definition of $Y$, and the property \eqref{goodplank1}, 
\begin{align} \label{muestimate} \int_Y \mu_j^{\frac{p}{p-2} } &\lesssim 2^{\frac{ 2j(3-\alpha)}{p-2} } \int_Y \mu_j \\
\notag &\leq 2^{\frac{ 2j(3-\alpha)}{p-2} } \sum_{Q \subseteq Y} \int_Q \mu_j \\
\notag &\lesssim  2^{\frac{ 2j(3-\alpha)}{p-2} } \left( \frac{1}{M} \right) \sum_{Q \subseteq Y} \sum_{\substack{T \in \mathbb{W} \\ 2T \cap Q \neq \emptyset } } \int_{Q \cap 3T} \mu_j \\
\notag &\lesssim  2^{\frac{ 2j(3-\alpha)}{p-2} }  \left( \frac{1}{M} \right)  \sum_{T \in \mathbb{W}} \int_{3T} \mu_j \\
\notag &\lesssim  2^{\frac{ 2j(3-\alpha)}{p-2} }  \left( \frac{1}{M} \right) \sum_{T \in \mathbb{W}} \mu(4T) + 2^{-100j} \\
\notag &\lesssim 2^{\frac{2j(3-\alpha)}{p-2} - j\left(1-\epsilon_0 \right)} \left( \frac{ \left\lvert \mathbb{W} \right\rvert }{M} \right) N^{-1}. \end{align}
This bounds the second factor in \eqref{pause}, so it remains to bound the first factor. 

By rescaling by $2^j$, applying the refined decoupling inequality (see Theorem~\ref{refineddecouplingtheorem}), and then rescaling back,
\begin{equation} \label{afterdecoupling} \left\lVert \sum_{T \in \mathbb{W}} f_T \right\rVert_{L^p(Y)} \lesssim 2^{j \zeta} \left( \frac{M}{\left\lvert \mathbb{W} \right\rvert} \right)^{\frac{1}{2} - \frac{1}{p} } \left( \sum_{T \in \mathbb{W} } \left\lVert f_T\right\rVert_p^2 \right)^{1/2}. \end{equation}
By recalling the definition \eqref{fTdefn} of the $f_T$'s and applying the Hausdorff-Young inequality, followed by Hölder's inequality, 
\[ \left\lVert f_T\right\rVert_p \lesssim \left\lVert M_T \mu\right\rVert_2 2^{\frac{3j}{2} \left( \frac{1}{2} - \frac{1}{p}\right)}. \]
Substituting into \eqref{afterdecoupling} gives
\begin{equation} \label{refineddecoupling} \left\lVert \sum_{T \in \mathbb{W}} f_T \right\rVert_{L^p(Y)} \lesssim 2^{\frac{3j}{2} \left( \frac{1}{2} - \frac{1}{p} + \zeta \right)} \left( \frac{M}{\left\lvert \mathbb{W} \right\rvert} \right)^{\frac{1}{2} - \frac{1}{p} } \left( \sum_{T \in \mathbb{T}_{j,g} } \left\lVert M_{T} \mu\right\rVert_2^2 \right)^{1/2}. \end{equation}
Substituting \eqref{refineddecoupling} and \eqref{muestimate} into \eqref{pause} and then \eqref{mucauchyschwarz} gives (after some algebra)
\begin{multline*} \sum_{T \in \mathbb{T}_{j,g} } \left\lVert M_{T} \mu\right\rVert_2^2 \\
\lesssim N^{-\frac{p-2}{2p}} 2^{j \left[ \frac{5-2\alpha}{2p} + \frac{1}{4} +\frac{\epsilon_0(p-2)}{2p} + \frac{3\zeta}{2}\right]} \left( \sum_{T \in \mathbb{T}_{j,g} } \left\lVert M_{T} \mu\right\rVert_2^2 \right)^{1/2} \mu(\mathbb{R}^3)^{1/2}. \end{multline*}
This gives, by cancelling the common factor,
\begin{equation} \label{triplestar} \sum_{T \in \mathbb{T}_{j,g} } \left\lVert M_{T} \mu\right\rVert_2^2 \lesssim  N^{-\frac{p-2}{p}} 2^{j \left[ \frac{5-2\alpha}{p} + \frac{1}{2} +\frac{\epsilon_0(p-2)}{p} + 3\zeta  \right]}\mu(\mathbb{R}^3). \end{equation}
By taking $p=6$ and recalling that $\alpha = 3-2\beta$, this simplifies to
\[ \sum_{T \in \mathbb{T}_{j,g} } \int_{\mathbb{R}^3} \left\lvert M_{T} \mu\right\rvert^2 \lesssim N^{-2/3} 2^{j\left( \frac{1+2\beta}{3}+ \frac{2\epsilon_0}{3} + 3\zeta \right)}\mu(\mathbb{R}^3), \]
which verifies \eqref{enough} and therefore concludes the proof of Lemma~\ref{energylemma}.   
\end{proof} 

The proof of Theorem~\ref{projmeasure} will be very similar to the proof of the Lemma~\ref{energylemma}.
\begin{proof}[Proof of Theorem~\ref{projmeasure}]  As in the proof of Lemma~\ref{energylemma}, it may be assumed that $\gamma$ is restricted to an interval of diameter at most $r=r(\gamma)$ on which Lemma~\ref{IBP} holds.  It may also be assumed that $\alpha \leq 3$, $c_{\alpha}(\mu) \leq 1$ and that $\mu$ has support in the unit ball. Since the exceptional set is Borel measurable (see \cite{harris24}), by Frostman's lemma it suffices to prove that for any $\beta >(3-\alpha)/2$, $\rho_{\theta\sharp} \mu \ll \mathcal{H}^1$ for $\lambda$-a.e.~$\theta \in I$, whenever $\lambda$ is a Borel measure supported in $I$ with $c_{\beta}(\lambda) < \infty$. Therefore, let $\beta$ be such that $\beta > (3-\alpha)/2$, and let $\lambda$ be a Borel measure supported in $I$ with $c_{\beta}(\lambda) < \infty$. Let $\epsilon>0$ be such that $\epsilon \ll \beta - \frac{3-\alpha}{2}$. Choose $\sigma>0$ such that $\sigma \ll \min\left\{ \epsilon, \delta_{\epsilon} \right\}$, where $\delta_{\epsilon}$ is an exponent corresponding to $\epsilon$ from Lemma~\ref{energylemma}. Using Definition~\ref{decompdefn}, define $\mu_b$ by 
\[  \mu_b = \sum_{j \geq 1} \sum_{\tau \in \Lambda_j} \sum_{S \in \mathbb{S}_{\tau,b} } M_S \mu, \]
where, for each $j \geq 1$ and $\tau \in \Lambda_j$, the set of ``bad'' slabs corresponding to $\tau$ is defined by 
\[ \mathbb{S}_{\tau, b} = \left\{ S \in \mathbb{S}_{\tau} : \mu(100S) \geq 2^{-j(1-\epsilon)} \right\}. \]
For $\lambda$-a.e.~$\theta \in I$,
\begin{equation} \label{pushbaddefn} \rho_{\theta \sharp}  \mu_b  := \sum_{j \geq 1} \sum_{\tau \in \Lambda_j} \sum_{S \in \mathbb{S}_{\tau,b} } \rho_{\theta \sharp} M_S\mu, \end{equation}
where, for $\lambda$-a.e.~$\theta \in I$, the series will be shown to be absolutely convergent in $L^1(\mathcal{H}^1)$. The $\lambda$-a.e.~absolute convergence of \eqref{pushbaddefn} in $L^1(\mathcal{H}^1)$ will imply that $\rho_{\theta\sharp} \mu_b \in L^1(\mathcal{H}^1)$ for $\lambda$-a.e.~$\theta \in I$ and that the series is $\lambda$-a.e.~well-defined as an $L^1(\mathcal{H}^1)$ limit. Define 
\[ \rho_{\theta\sharp} \mu_g = \rho_{\theta\sharp} \mu - \rho_{\theta \sharp} \mu_b, \]
for each $\theta \in I$ such that the sum defining $\rho_{\theta\sharp}\mu_b$ converges in $L^1(\mathcal{H}^1)$. It will be shown that 
\[ \rho_{\theta \sharp} \mu_g \in L^2(\mathcal{H}^1), \]
for $\lambda$-a.e.~$\theta \in I$. Together with $\rho_{\theta\sharp} \mu_b \in L^1(\mathcal{H}^1)$ for $\lambda$-a.e.~$\theta \in I$, this will imply that $\rho_{\theta \sharp}\mu \in L^1(\mathcal{H}^1)$ (or equivalently $\rho_{\theta \sharp}\mu \ll \mathcal{H}^1$) for $\lambda$-a.e.~$\theta \in I$.

It will first be shown that
\begin{equation} \label{pause333} \int \sum_{j \geq 1} \sum_{\tau \in \Lambda_j} \sum_{S \in \mathbb{S}_{\tau,b} } \left\lVert \rho_{\theta \sharp}  M_S\mu \right\rVert_{L^1(\mathcal{H}^1 ) } \, d\lambda(\theta) \lesssim c_{\beta}(\lambda)^{1/2} \lambda(\mathbb{R})^{1/2} \mu(\mathbb{R}^3). \end{equation}
The proof of this is similar to the proof of Lemma~\ref{energylemma}, but the details will be sketched. The left-hand side of \eqref{pause333} can be written as
\begin{align}  \notag &\sum_{j \geq 1}   \int \sum_{\tau \in \Lambda_j} \sum_{S \in \mathbb{S}_{\tau,b} }\left\lVert \rho_{\theta \sharp}  M_S\mu\right\rVert_{L^1(\mathcal{H}^1) } \, d\lambda(\theta)  \\
\label{mainpart2} &\quad = \sum_{j \geq 1}  \int  \sum_{\substack{\tau \in \Lambda_j: \\ \left\lvert \theta_{\tau} - \theta \right\rvert \leq 2^{-j\left(1-\sigma\right)}}}  \sum_{S \in \mathbb{S}_{\tau,b} }\left\lVert \rho_{\theta \sharp} M_S\mu\right\rVert_{L^1(\mathcal{H}^1) } \, d\lambda(\theta)\\
\label{negligible2} &\qquad + \sum_{j \geq 1}   \int  \sum_{\substack{\tau \in \Lambda_j: \\ \left\lvert \theta_{\tau} - \theta \right\rvert > 2^{-j\left(1-\sigma\right)}}} \sum_{S \in \mathbb{S}_{\tau,b} }\left\lVert \rho_{\theta \sharp} M_S\mu \right\rVert_{L^1(\mathcal{H}^1) } \, d\lambda(\theta). \end{align}
By Lemma~\ref{IBP},
\[ \eqref{negligible2}  \lesssim \mu(\mathbb{R}^3)\lambda(\mathbb{R}).\] 
By Lemma~\ref{MTf},
\begin{align*} \eqref{mainpart2} &\leq  \sum_{j \geq 1}  \int \sum_{\substack{\tau \in \Lambda_j: \\ \left\lvert \theta_{\tau} - \theta \right\rvert \leq 2^{-j\left(1-\sigma\right)}}}  \sum_{S \in \mathbb{S}_{\tau,b} }\left\lVert  M_S\mu \right\rVert_{L^1(\mathbb{R}^3) } \, d\lambda(\theta) \\
&\lesssim \mu(\mathbb{R}^3)\lambda(\mathbb{R}) + \sum_{j \geq 1} 2^{3j \sigma} \int  \sum_{\substack{\tau \in \Lambda_j: \\ \left\lvert \theta_{\tau} - \theta \right\rvert \leq 2^{-j\left(1-\sigma\right)}}}  \sum_{S \in \mathbb{S}_{\tau,b} }\mu(2S) \, d\lambda(\theta). \end{align*}
As in the proof of Lemma~\ref{energylemma}, the non-tail term satisfies 
\[ \sum_{j \geq 1} 2^{3j \sigma} \int  \sum_{\substack{\tau \in \Lambda_j: \\ \left\lvert \theta_{\tau} - \theta \right\rvert \leq 2^{-j\left(1-\sigma\right)}}}  \sum_{S \in \mathbb{S}_{\tau,b} }\mu(2S) \, d\lambda(\theta) \lesssim  \sum_{j \geq 1} 2^{10j \sigma}\int \mu(B_j(\theta) ) \, d\lambda(\theta),  \]
where, for each $\theta \in I$ and each $j$, 
\[ B_j(\theta) =  \bigcup_{\tau \in \Lambda_j:  \left\lvert \theta_{\tau} - \theta \right\rvert \leq 2^{-j\left(1-\sigma\right)}}  \bigcup_{S \in \mathbb{S}_{\tau,b} } 2S. \]
For any $\theta$, the set of possible $S$ occurring above has cardinality $\lesssim \mu(\mathbb{R}^3)2^{j(1-\epsilon)} 2^{2j \sigma}$; by the bounded overlap and the definition of the ``bad'' slabs. Moreover, since $\alpha > 3-2\beta$ and $\mu$ is supported in $B_3(0,1)$, 
\[ c_{3-2\beta}(\mu) \lesssim c_{\alpha}(\mu). \]
Hence, by Lemma~\ref{energylemma}, for each $j \geq 1$,
\begin{multline} \label{geometricseries} \int \mu(B_j(\theta) ) \, d\lambda(\theta) \leq \int \left( \rho_{\theta\sharp} \mu \right) \left(\bigcup_{\tau \in \Lambda_j : \left\lvert \theta_{\tau} - \theta \right\rvert \leq 2^{-j\left(1-\sigma\right)}} \bigcup_{ S \in \mathbb{S}_{\tau,b}} 2 \rho_{\theta}(S) \right)  \, d\lambda(\theta) \\
\lesssim  c_{\beta}(\lambda)^{1/2} \lambda(\mathbb{R})^{1/2} 2^{j \left(-\delta_{\epsilon}/2 + 10 \sigma \right)} \mu(\mathbb{R}^3). \end{multline} Since $\sigma \ll \delta_{\epsilon}$, summing the above inequality over $j$ gives
\[ \eqref{mainpart2} \lesssim  c_{\beta}(\lambda)^{1/2} \lambda(\mathbb{R})^{1/2}\mu(\mathbb{R}^3). \]

It remains to show that $\rho_{\theta\sharp}  \mu_g \in L^2(\mathcal{H}^1)$ for $\lambda$-a.e.~$ \theta \in I$. To prove this, by Plancherel's theorem in 1 dimension it suffices to show that
\begin{equation} \label{quadriplestar} \int  \int_{\mathbb{R}} \left\lvert \widehat{\mu_g} \left( t \gamma(\theta) \right) \right\rvert^2 \, dt \,  d\lambda(\theta) \lesssim c_{\beta}(\lambda) \mu(\mathbb{R}^3). \end{equation}
Since $\mu_g$ has not strictly been defined, the above is ambiguous, but can be made precise as follows. By the a.e.~definition of the complex measure $\rho_{\theta\sharp}  \mu_g$, 
\[ \int_0^{2\pi} \int \left\lvert \widehat{\rho_{\theta \sharp} \mu_g}(\xi)\right\rvert^2 \, d\mathcal{H}^1(\xi) \, d\theta = \int_0^{2\pi} \int \left\lvert \widehat{\rho_{\theta \sharp} \mu}(\xi)-\widehat{\rho_{\theta \sharp} \mu_b}(\xi)\right\rvert^2 \, d\mathcal{H}^1(\xi) \, d\theta, \]
where $\rho_{\theta \sharp} \mu$ and $\rho_{\theta \sharp}\mu_b$ are well-defined complex measures in $L^1$ for a.e.~$\theta$, the Fourier transform above is the $L^1$ Fourier transform, and by \cite[Theorem~3.3]{mattila2} it suffices to show that the above quantity is finite. Since the $L^1$ Fourier transform is continuous on $L^1$, and since the sum defining $\rho_{\theta\sharp}\mu_b$ converges in $L^1$, by the Fourier slice theorem, this is equivalent to showing that 
\[ \int_0^{2\pi} \int \left\lvert \widehat{\mu}(t\gamma(\theta))- \sum_{j \geq 1} \sum_{\tau \in \Lambda_j} \sum_{S \in \mathbb{S}_{\tau,b} } \widehat{ M_S\mu}(t\gamma(\theta))\right\rvert^2 \, d\mathcal{H}^1(\xi) \, d\theta< \infty. \]
By abusing notation and writing (from now on)
\[ \widehat{\mu_g} =  \widehat{\mu} - \sum_{j \geq 1} \sum_{\tau \in \Lambda_j} \sum_{S \in \mathbb{S}_{\tau,b} } \widehat{ M_S\mu}, \]
where the sum trivially converges pointwise everywhere by rapid decay, by symmetry it is enough to show that for any $j \geq 1$,  
\[ \int  \int_{2^{j-1}}^{2^j}\left\lvert \widehat{\mu_g}\left( t \gamma(\theta) \right) \right\rvert^2 \, dt \,  d\lambda(\theta) \lesssim 2^{-j\epsilon} c_{\beta}(\lambda) \mu(\mathbb{R}^3).   \]
By similar reasoning (to \eqref{pause29}) justifying that \eqref{enough} suffices in Lemma~\ref{energylemma}, it suffices to show that for any $j \geq 1$,
\begin{equation} \label{enough2} \sum_{T \in \mathbb{T}_{j, g}} \int_{\mathbb{R}^3} \left\lvert M_{T} \mu\right\rvert^2 \, dx \lesssim N^{-2/3} 2^{j\left( \frac{1+2\beta}{3}- 10 \epsilon \right) } \mu(\mathbb{R}^3), \end{equation}
the main difference from \eqref{enough} being the negative sign in front of the $\epsilon$. Here $\{\phi\}_{\phi}$ is a partition of the cone at distance $\sim 2^j$ from the origin into standard boxes of dimensions $\sim 2^j \times 2^{j/2} \times 1$, and $\mathbb{T}_{j,g} = \bigcup_{\phi \in \Phi_j } \bigcup \mathbb{T}_{\phi, \kappa} $. The parameter $\kappa$ is a fixed dyadic number, and for each $\phi$, the set $\mathbb{T}_{\phi, \kappa}$ is a subset of the cover $\mathbb{T}_{\phi}$ of $\mathbb{R}^3$ by a boundedly overlapping set of planks of dimensions $2^{j\left(-1+\sigma \right) } \times 2^{j\left( -1/2 + \sigma \right) } \times 2^{j\sigma }$ dual to $\phi$, with $\mu(4T) \sim \kappa$ for all $T \in \bigcup_{\phi} \mathbb{T}_{\phi, \kappa}$, and each $T$ satisfies \eqref{goodplank1} with $N= N(j)$, but with $\epsilon_0$ replaced by $\epsilon$. 

By a similar argument to the justification of \eqref{triplestar} in the proof of Lemma~\ref{energylemma},
\[ \sum_{T \in \mathbb{T}_{j,g}} \int_{\mathbb{R}^3} \left\lvert M_{T} \mu\right\rvert^2 \lesssim  N^{-2/3} 2^{j \left[ \frac{5-2\alpha}{p} + \frac{1}{2} +100\epsilon \right]}\mu(\mathbb{R}^3), \]
where $p=6$. Since $\alpha >3-2\beta$, $p=6$ and $\epsilon \ll \beta - \frac{3-\alpha}{2}$, this implies that
\[ \sum_{T \in \mathbb{T}_{j,g}} \int_{\mathbb{R}^3} \left\lvert M_T \mu\right\rvert^2 \lesssim  N^{-2/3}2^{j\left(\frac{1+2\beta}{3}-10\epsilon\right)}\mu(\mathbb{R}^3), \]
which verifies \eqref{enough2} and finishes the proof of Theorem~\ref{projmeasure}. \end{proof}

Only a very brief sketch of the proof of Theorem~\ref{Lptheorem} will be given, since it is very similar to the proof of Theorem~\ref{projmeasure}. 

\begin{proof}[Proof of Theorem~\ref{Lptheorem}] 
 The notation and setup will be almost the same as in the proof of Theorem~\ref{projmeasure}. It may be assumed that $\gamma$ is restricted to an interval of diameter at most $r=r(\gamma)$ on which Lemma~\ref{IBP} holds.  It may also be assumed that $\alpha \leq 3$ and $c_{\alpha}(\mu) \leq 1$.  Let $\lambda$ be a Borel measure supported in $I$ with $c_{\beta}(\lambda) < \infty$.  By scaling it may be assumed that $c_{\alpha}(\mu) = 1$.

Let $\epsilon>0$ be such that $\epsilon \ll \beta - \frac{3-\alpha}{2}$. Choose $\sigma>0$ such that $\sigma \ll \min\left\{ \epsilon, \delta_{\epsilon} \right\}$, where $\delta_{\epsilon}$ is an exponent corresponding to $\epsilon$ from Lemma~\ref{energylemma}. Using Definition~\ref{decompdefn}, define $\mu_b$ by 
\begin{equation} \label{mubsum}  \mu_b = \sum_{j \geq 1} \sum_{\tau \in \Lambda_j} \sum_{S \in \mathbb{S}_{\tau,b} } M_S \mu, \end{equation}
where, for each $j \geq 1$ and $\tau \in \Lambda_j$, 
\[ \mathbb{S}_{\tau, b} = \left\{ S \in \mathbb{S}_{\tau} : \mu(100S) \geq 2^{-j(1-\epsilon)} \right\}. \]

 One simplification compared to the proof of Theorem~\ref{projmeasure} is that, since $p>1$ (where $p$ will be chosen in a moment), the dual of $L^p$ has a dense subspace of continuous functions, which means that, by approximation and a simple duality argument, it suffices to prove Theorem~\ref{Lptheorem} for $\mu$ a positive smooth function supported in the unit ball. This simplifies the convergence (e.g.~in the Schwartz space) in the sum \eqref{mubsum} defining $\mu_b$, and implies the trivial identity
\begin{equation} \label{trivial} \int \left\lVert \rho_{\theta \sharp} \mu \right\rVert_{L^1(\mathcal{H}^1) } \, d\lambda(\theta)  = \lambda(\mathbb{R}) \mu(\mathbb{R}^3) \lesssim  \lambda(\mathbb{R})^{1/2} c_{\beta}(\lambda)^{1/2} \mu(\mathbb{R}^3), \end{equation}
since any measure absolutely continuous with respect to Lebesgue measure has $L^1$ norm equal to its total mass. 
Define 
\[  \mu_g =  \mu -  \mu_b. \]
Let $\zeta>0$ be a small parameter to be chosen. By the triangle inequality followed by Hölder's inequality, to prove that 
\begin{equation} \label{pause3335} \int  \left\lVert  \rho_{\theta \sharp}  \mu_b \right\rVert_{L^p(\mathcal{H}^1 ) }^p \, d\lambda(\theta) \lesssim \lambda(\mathbb{R})^{1/2} c_{\beta}(\lambda)^{1/2}  \mu(\mathbb{R}^3), \end{equation}
 it suffices to show that
\begin{equation} \label{pause3334} \int \sum_{j \geq 1}  \left\lVert \sum_{\tau \in \Lambda_j} \sum_{S \in \mathbb{S}_{\tau,b} } \rho_{\theta \sharp}  M_S\mu \right\rVert_{L^p(\mathcal{H}^1 ) }^p 2^{\zeta j } \, d\lambda(\theta) \lesssim  \lambda(\mathbb{R})^{1/2} c_{\beta}(\lambda)^{1/2} \mu(\mathbb{R}^3). \end{equation}
The first step is the (blunt) inequality
\begin{multline*}  \left\lVert \sum_{\tau \in \Lambda_j} \sum_{S \in \mathbb{S}_{\tau,b} }  \rho_{\theta \sharp}  M_S\mu\right\rVert_{L^p(\mathcal{H}^1 ) }^p \lesssim \left\lVert \sum_{\tau \in \Lambda_j} \sum_{S \in \mathbb{S}_{\tau,b} }  \rho_{\theta \sharp}  M_S\mu \right\rVert_{L^1(\mathcal{H}^1 ) } 2^{10^{10}j(p-1)} \\
\leq \sum_{\tau \in \Lambda_j} \sum_{S \in \mathbb{S}_{\tau,b} }\left\lVert   \rho_{\theta \sharp}  M_S\mu \right\rVert_{L^1(\mathcal{H}^1 ) } 2^{10^{10}j(p-1)}. \end{multline*}
Inserting into \eqref{pause3334}, it suffices to show that 
\begin{equation} \label{pause3336} \int \sum_{j \geq 1}  \left\lVert \sum_{\tau \in \Lambda_j} \sum_{S \in \mathbb{S}_{\tau,b} } \rho_{\theta \sharp}  M_S\mu \right\rVert_{L^1(\mathcal{H}^1 ) } 2^{\zeta j } 2^{10^{10}j(p-1)}  \, d\lambda(\theta) \lesssim  \lambda(\mathbb{R})^{1/2} c_{\beta}(\lambda)^{1/2} \mu(\mathbb{R}^3). \end{equation}
The only difference between this bound and the bound \eqref{pause333} from the proof of Theorem~\ref{projmeasure} is the factor $2^{\zeta j } 2^{10^{10}j(p-1)}$. But if the same proof of \eqref{pause333} is followed up to \eqref{geometricseries}, the decay factor $2^{j \left(-\delta_{\epsilon}/2 + 10 \sigma \right)}$ multiplied by the extra factor will still result in a decaying geometric series in $j$, by taking $p$ very close to 1, and 
by taking $\zeta$ very small. This proves \eqref{pause3336} and therefore \eqref{pause3334}. The bound \eqref{pause3336} also implies that
\[ \int \left\lVert \rho_{\theta \sharp} \mu_b \right\rVert_{L^1(\mathcal{H}^1) } \, d\lambda(\theta)   \lesssim  \lambda(\mathbb{R})^{1/2} c_{\beta}(\lambda)^{1/2} \mu(\mathbb{R}^3). \]
Combining with \eqref{trivial}, using the triangle inequality, yields 
\begin{equation} \label{trivialgood} \int \left\lVert \rho_{\theta \sharp} \mu_g \right\rVert_{L^1(\mathcal{H}^1) } \, d\lambda(\theta)  \lesssim  \lambda(\mathbb{R})^{1/2} c_{\beta}(\lambda)^{1/2} \mu(\mathbb{R}^3). \end{equation}
Since there is already a bound on the $L^2$ norm of the ``good'' part from the proof of Theorem~\ref{projmeasure} (see \eqref{quadriplestar}):
\[ \int \left\lVert \rho_{\theta \sharp} \mu_g \right\rVert_{L^2(\mathcal{H}^1) }^2 \, d\lambda(\theta)  \lesssim   c_{\beta}(\lambda) \mu(\mathbb{R}^3), \]
interpolating with \eqref{trivialgood} using Hölder's inequality yields
\[ \int \left\lVert \rho_{\theta \sharp} \mu_g \right\rVert_{L^p(\mathcal{H}^1) }^p \, d\lambda(\theta)  \lesssim  \lambda(\mathbb{R})^{\frac{2-p}{2}} c_{\beta}(\lambda)^{\frac{p}{2}} \mu(\mathbb{R}^3). \]
Combining with \eqref{pause3335}, using that $\lambda(\mathbb{R})^{1/2} c_{\beta}(\lambda)^{1/2} \lesssim \lambda(\mathbb{R})^{\frac{2-p}{2}} c_{\beta}(\lambda)^{\frac{p}{2}}$ for $1 \leq p \leq 2$, gives 
\[ \int \left\lVert \rho_{\theta \sharp} \mu \right\rVert_{L^p(\mathcal{H}^1) }^p \, d\lambda(\theta)  \lesssim  \lambda(\mathbb{R})^{\frac{2-p}{2}} c_{\beta}(\lambda)^{\frac{p}{2}} \mu(\mathbb{R}^3). \]
This finishes the proof of Theorem~\ref{Lptheorem}. \end{proof}

\appendix
\section{Refined decoupling}

The inequality in Theorem~\ref{refineddecouplingtheorem} below is the refined decoupling inequality for cones, from \cite{GGGHMW}. Although the proof from \cite{GGGHMW} appears rather technical, the idea is the same as in the parabola case in \cite{GIOW}. The proof is by induction on scales, and the theorem at one scale follows from the theorem at a larger scale by the standard decoupling theorem and Lorentz rescaling (the analogue of parabolic rescaling). 
\begin{theorem}[{\cite[Theorem~9]{GGGHMW}}] \label{refineddecouplingtheorem} Let $I$ be a compact interval, and let $\gamma :I \to S^2$ be a $C^2$ unit speed curve with $\det(\gamma, \gamma', \gamma'' )$ nonvanishing on $I$. Then if $c>0$ is sufficiently small (depending only on $\gamma$), then for any $\epsilon >0$, there exists $\delta_0>0$ such that the following holds for all $0 < \delta < \delta_0$, and any $R \geq 1$. Let $\Theta_R$ be a maximal $c R^{-1/2}$-separated subset of $I$, and for each $\theta \in \Theta_R$, let
\begin{multline*} \tau(\theta) := \\
\left\{ \lambda_1 \gamma(\theta) + \lambda_2 \gamma'(\theta) + \lambda_3 (\gamma \times \gamma')(\theta) : 1/2 \leq \lambda_1 \leq 1, \lvert \lambda_2 \rvert \leq R^{-1/2}, \lvert \lambda_3 \rvert \leq R^{-1} \right\}. \end{multline*}
For each $\tau = \tau(\theta)$, let $\mathbb{T}_{\tau}$ be a  $\sim 1$-overlapping cover of $\mathbb{R}^3$ by translates of 
\[ \left\{ \lambda_1 \gamma(\theta) + \lambda_2 \gamma'(\theta) + \lambda_3 (\gamma \times \gamma')(\theta) : \lvert \lambda_1 \rvert \leq R^{\delta}, \lvert \lambda_2 \rvert \leq R^{1/2 + \delta}, \lvert \lambda_3 \rvert \leq R^{1+\delta} \right\}. \]
If $2 \leq p \leq 6$, and 
\[ \mathbb{W} \subseteq \bigcup_{\theta \in \Theta_R} \mathbb{T}_{\tau(\theta)}, \]
and 
\[ \sum_{T \in \mathbb{W} } f_T \]
is such that $\lVert f_T \rVert_p$ is constant over $T \in \mathbb{W}$ up to a factor of 2, with $\supp \widehat{f_T} \subseteq \tau(T)$ and 
\[ \lVert f_T\rVert_{L^{\infty}(B(0,R) \setminus T) } \leq A R^{-10000} \lVert f_T \rVert_p, \]
and $Y$ is a disjoint union of balls in $B_3(0,R)$ of radius 1, such that each ball $Q \subseteq Y$ intersects at most $M$ planks $2T$ with $T \in \mathbb{W}$, then 
\[ \left\lVert \sum_{T \in \mathbb{W} } f_T \right\rVert_{L^p(Y) } \leq C_{A, \gamma,c, \epsilon, \delta} R^{\epsilon} \left( \frac{M}{\left\lvert \mathbb{W} \right\rvert } \right)^{\frac{1}{2} - \frac{1}{p} } \left( \sum_{T \in \mathbb{W} } \left\lVert f_T \right\rVert_p^2 \right)^{1/2}. \] \end{theorem} 
 The statement in \cite{GGGHMW} assumed dyadic constancy of $\|f_T\|_2$ rather than $\|f_T\|_p$, but this was a typo; the proof actually refers to dyadic constancy of $\|f_T\|_p$ \cite[p.~956]{GGGHMW}. By the uncertainty principle, these assumptions are interchangeable in the sense that the theorem under either assumption can be deduced from the other. 

\bibliographystyle{plainurl}
\bibliography{exceptional}

@article {besicovitchmoran,
    AUTHOR = {Besicovitch, A. S. and Moran, P. A. P.},
     TITLE = {The measure of product and cylinder sets},
   JOURNAL = {J. London Math. Soc.},
  FJOURNAL = {The Journal of the London Mathematical Society},
    VOLUME = {20},
      YEAR = {1945},
     PAGES = {110--120},
      ISSN = {0024-6107,1469-7750},
   MRCLASS = {27.2X},
  MRNUMBER = {16448},
MRREVIEWER = {H.\ Federer},
       DOI = {10.1112/jlms/s1-20.2.110},
       URL = {https://doi.org/10.1112/jlms/s1-20.2.110},
}

@misc{KR,
      title={Finding subsets of positive measure}, 
      author={Bjørn Kjos-Hanssen and Jan Reimann},
      year={2014},
      eprint={1408.1999},
      archivePrefix={arXiv},
      primaryClass={math.LO},
      url={https://arxiv.org/abs/1408.1999}, 
}

@article {falconer,
    AUTHOR = {Falconer, K. J.},
     TITLE = {Hausdorff dimension and the exceptional set of projections},
   JOURNAL = {Mathematika},
  FJOURNAL = {Mathematika. A Journal of Pure and Applied Mathematics},
    VOLUME = {29},
      YEAR = {1982},
    NUMBER = {1},
     PAGES = {109--115},
      ISSN = {0025-5793},
   MRCLASS = {28A75 (53C65)},
  MRNUMBER = {673510},
MRREVIEWER = {K.\ E.\ Hirst},
       DOI = {10.1112/S0025579300012201},
       URL = {https://doi.org/10.1112/S0025579300012201},
}

@article {kaufman,
    AUTHOR = {Kaufman, Robert},
     TITLE = {On {H}ausdorff dimension of projections},
   JOURNAL = {Mathematika},
  FJOURNAL = {Mathematika. A Journal of Pure and Applied Mathematics},
    VOLUME = {15},
      YEAR = {1968},
     PAGES = {153--155},
      ISSN = {0025-5793},
   MRCLASS = {54.70 (26.00)},
  MRNUMBER = {248779},
MRREVIEWER = {F.\ Raymond},
       DOI = {10.1112/S0025579300002503},
       URL = {https://doi.org/10.1112/S0025579300002503},
}

@article {davies,
    AUTHOR = {Davies, R. O.},
     TITLE = {Subsets of finite measure in analytic sets},
      NOTE = {Nederl. Akad. Wetensch. Proc. Ser. A {\bf 55}},
   JOURNAL = {Indag. Math.},
  FJOURNAL = {},
    VOLUME = {14},
      YEAR = {1952},
     PAGES = {488--489},
   MRCLASS = {27.2X},
  MRNUMBER = {53184},
MRREVIEWER = {H.\ D.\ Ursell},
}

@article {greenharrisou,
    AUTHOR = {Green, John and Harris, Terence L. J. and Ou, Yumeng},
     TITLE = {An {$L^{3/2}$} {$\mathrm{SL}_2$} {K}akeya maximal inequality},
   JOURNAL = {Anal. PDE},
  FJOURNAL = {Analysis \& PDE},
    VOLUME = {19},
      YEAR = {2026},
    NUMBER = {7},
     PAGES = {1295--1338},
    ISSN = {},
   MRCLASS = {},
  MRNUMBER = {},
       DOI = {10.2140/apde.2026.19.1295},
       URL = {https://doi.org/10.2140/apde.2026.19.1295},
}

@article {guth,
    AUTHOR = {Guth, Larry},
     TITLE = {Restriction estimates using polynomial partitioning {II}},
   JOURNAL = {Acta Math.},
  FJOURNAL = {Acta Mathematica},
    VOLUME = {221},
      YEAR = {2018},
    NUMBER = {1},
     PAGES = {81--142},
}

@misc{pramanikyangzahl,
      title={A {F}urstenberg-type problem for circles, and a {K}aufman-type restricted projection theorem in $\mathbb{R}^3$}, 
      author={Malabika Pramanik and Tongou Yang and Joshua Zahl},
      year={2024},
      eprint={2207.02259},
      archivePrefix={arXiv},
      primaryClass={math.CA},
      url={https://arxiv.org/abs/2207.02259}, 
}

@article {fasslerorponen,
    AUTHOR = {F\"{a}ssler, Katrin and Orponen, Tuomas},
     TITLE = {Vertical projections in the {H}eisenberg group via cinematic
              functions and point-plate incidences},
   JOURNAL = {Adv. Math.},
  FJOURNAL = {Advances in Mathematics},
    VOLUME = {431},
      YEAR = {2023},
     PAGES = {Paper No. 109248, 41},
}

@article {GGGHMW,
    AUTHOR = {Gan, Shengwen and Guo, Shaoming and Guth, Larry and Harris,
              Terence L. J. and Maldague, Dominique and Wang, Hong},
     TITLE = {On restricted projections to planes in {$\mathbb R^3$}},
   JOURNAL = {Amer. J. Math.},
  FJOURNAL = {American Journal of Mathematics},
    VOLUME = {148},
      YEAR = {2026},
    NUMBER = {4},
     PAGES = {915--960},
      ISSN = {0002-9327,1080-6377},
   MRCLASS = {42B15 (42B20 42B25)},
  MRNUMBER = {5115659},
}

@article {ganguowang,
    AUTHOR = {Gan, Shengwen and Guo, Shaoming and Wang, Hong},
     TITLE = {A restricted projection problem for fractal sets in {$\mathbb
              R^n$}},
   JOURNAL = {Camb. J. Math.},
  FJOURNAL = {Cambridge Journal of Mathematics},
    VOLUME = {12},
      YEAR = {2024},
    NUMBER = {3},
     PAGES = {535--561},
      ISSN = {2168-0930,2168-0949},
   MRCLASS = {28A75 (28A78 28A80)},
  MRNUMBER = {4806061},
MRREVIEWER = {Jonathan\ MacDonald\ Fraser},
       DOI = {10.4310/cjm.240930233934},
       URL = {https://doi.org/10.4310/cjm.240930233934},
}

@article {ganguthmaldague,
    AUTHOR = {Gan, Shengwen and Guth, Larry and Maldague, Dominique},
     TITLE = {An Exceptional Set Estimate for Restricted Projections to Lines in $\mathbb{R}^3$},
   JOURNAL = {J. Geom. Anal.},
    VOLUME = {34},
    NUMBER = {15},
      YEAR = {2024},
}

@article{mattila,
 author = {Mattila, Pertti},
 title = {Hausdorff dimension of plane sections and general intersections},
 fjournal = {Bulletin of the London Mathematical Society},
 journal = {Bull. Lond. Math. Soc.},
 issn = {0024-6093},
 volume = {56},
 number = {6},
 pages = {1988--1998},
 year = {2024},
 language = {English},
 doi = {10.1112/blms.13038},
 zbMATH = {7922699}
}

@book {wolff,
    AUTHOR = {Wolff, Thomas H.},
     TITLE = {Lectures on harmonic analysis},
    SERIES = {University Lecture Series},
    VOLUME = {29},
    EDITOR = {\L aba, Izabella and Shubin, Carol},
 PUBLISHER = {American Mathematical Society, Providence, RI},
      YEAR = {2003},
     PAGES = {x+137},
      ISBN = {0-8218-3449-5},
   MRCLASS = {42-02 (42B15)},
  MRNUMBER = {2003254},
MRREVIEWER = {Steven\ George\ Krantz},
       DOI = {10.1090/ulect/029},
       URL = {https://doi.org/10.1090/ulect/029},
}

@article {kov,
    AUTHOR = {K\"aenm\"aki, Antti and Orponen, Tuomas and Venieri, Laura},
     TITLE = {A {M}arstrand-type restricted projection theorem in {$\mathbb
              R^3$}},
   JOURNAL = {Amer. J. Math.},
  FJOURNAL = {American Journal of Mathematics},
    VOLUME = {147},
      YEAR = {2025},
    NUMBER = {1},
     PAGES = {81--123},
      ISSN = {0002-9327,1080-6377},
   MRCLASS = {28A80 (28A78)},
  MRNUMBER = {4861586},
MRREVIEWER = {Lars\ Olsen},
       DOI = {10.1353/ajm.2025.a950274},
       URL = {https://doi.org/10.1353/ajm.2025.a950274},
}

@article {liu,
    AUTHOR = {Liu, Bochen},
     TITLE = {Hausdorff dimension of pinned distance sets and the
              {$L^2$}-method},
   JOURNAL = {Proc. Amer. Math. Soc.},
  FJOURNAL = {Proceedings of the American Mathematical Society},
    VOLUME = {148},
      YEAR = {2020},
    NUMBER = {1},
     PAGES = {333--341},
      ISSN = {0002-9939,1088-6826},
   MRCLASS = {28A75 (42B20)},
  MRNUMBER = {4042855},
MRREVIEWER = {Simon\ Bortz},
       DOI = {10.1090/proc/14740},
       URL = {https://doi.org/10.1090/proc/14740},
}

@book {mattila1995,
    AUTHOR = {Mattila, Pertti},
     TITLE = {Geometry of sets and measures in {E}uclidean spaces},
    SERIES = {Cambridge Studies in Advanced Mathematics},
    VOLUME = {44},
      NOTE = {Fractals and rectifiability},
 PUBLISHER = {Cambridge University Press, Cambridge},
      YEAR = {1995},
     PAGES = {xii+343},
      ISBN = {0-521-46576-1; 0-521-65595-1},
   MRCLASS = {28A75 (49Q20)},
  MRNUMBER = {1333890},
MRREVIEWER = {Harold\ Parks},
       DOI = {10.1017/CBO9780511623813},
       URL = {https://doi.org/10.1017/CBO9780511623813},
}

@book {mattila2,
    AUTHOR = {Mattila, Pertti},
     TITLE = {Fourier analysis and {H}ausdorff dimension},
    SERIES = {Cambridge Studies in Advanced Mathematics},
    VOLUME = {150},
 PUBLISHER = {Cambridge University Press, Cambridge},
      YEAR = {2015},
     PAGES = {xiv+440},
      ISBN = {978-1-107-10735-9},
   MRCLASS = {28-02 (28A15 28A78 28A80 42B10 60J65)},
  MRNUMBER = {3617376},
MRREVIEWER = {Benjamin\ Steinhurst},
       DOI = {10.1017/CBO9781316227619},
       URL = {https://doi.org/10.1017/CBO9781316227619},
}

@article {GIOW,
    AUTHOR = {Guth, Larry and Iosevich, Alex and Ou, Yumeng and Wang, Hong},
     TITLE = {On {F}alconer's distance set problem in the plane},
   JOURNAL = {Invent. Math.},
  FJOURNAL = {Inventiones Mathematicae},
    VOLUME = {219},
      YEAR = {2020},
    NUMBER = {3},
     PAGES = {779--830},
      ISSN = {0020-9910,1432-1297},
   MRCLASS = {42B20 (28A80)},
  MRNUMBER = {4055179},
MRREVIEWER = {Jonathan\ MacDonald\ Fraser},
       DOI = {10.1007/s00222-019-00917-x},
       URL = {https://doi.org/10.1007/s00222-019-00917-x},
}

@incollection {harris24,
    AUTHOR = {Harris, Terence L. J.},
     TITLE = {Length of sets under restricted families of projections onto
              lines},
 BOOKTITLE = {Recent developments in harmonic analysis and its applications},
    SERIES = {Contemp. Math.},
    VOLUME = {792},
     PAGES = {1--17},
 PUBLISHER = {Amer. Math. Soc., [Providence], RI},
      YEAR = {[2024] \copyright 2024},
      ISBN = {978-1-4704-7140-8},
   MRCLASS = {28A78 (28A80)},
  MRNUMBER = {4694174},
       DOI = {10.1090/conm/792/15892},
       URL = {https://doi.org/10.1090/conm/792/15892},
}

@book {tao,
    AUTHOR = {Tao, Terence},
     TITLE = {An introduction to measure theory},
    SERIES = {Graduate Studies in Mathematics},
    VOLUME = {126},
 PUBLISHER = {American Mathematical Society, Providence, RI},
      YEAR = {2011},
     PAGES = {xvi+206},
      ISBN = {978-0-8218-6919-2},
   MRCLASS = {28-01},
  MRNUMBER = {2827917},
MRREVIEWER = {Mahendra\ G.\ Nadkarni},
       DOI = {10.1090/gsm/126},
       URL = {https://doi.org/10.1090/gsm/126},
}

\end{document}